\documentclass[12pt]{amsart}
\usepackage{amsfonts, amsmath,amssymb,amsthm}
\usepackage{mathrsfs}
\usepackage{graphics}
\usepackage[hmargin=3cm,vmargin=3cm]{geometry}
\usepackage[english]{babel}
\usepackage[latin1]{inputenc}
\usepackage[shortlabels]{enumitem}
\setlist[itemize]{leftmargin=20pt}
\usepackage{cite}
\allowdisplaybreaks
\usepackage[colorlinks, citecolor = blue]{hyperref}
\usepackage[color=green!40]{todonotes}
\usepackage{comment}

\newtheorem{TheoremLetter}{Theorem}
{}
\newtheorem{theorem}{Theorem}[section]
\newtheorem{lemma}[theorem]{Lemma}
\newtheorem{corollary}[theorem]{Corollary}
\newtheorem{prop}[theorem]{Proposition}
\newtheorem{definition}[theorem]{Definition}

\newtheorem{proposition}[theorem]{Proposition}
\newtheorem*{lemma*}{Lemma}

\newtheorem{remark}[theorem]{\textbf{Remark}}

\newcommand{\F}{\mathbb{F}}
\newcommand{\Z}{\mathbb{Z}}

\newcommand{\R}{\mathbb{R}}
\newcommand{\C}{\mathbb{C}}
\newcommand{\D}{\mathcal{D}}
\newcommand{\calQ}{\mathcal{Q}}

\DeclareMathOperator*{\esssup}{ess \sup}
\DeclareMathOperator*{\essinf}{ess \inf}

\newcommand{\vf}{\mathbf{f}}
\newcommand{\vg}{\mathbf{g}}
\newcommand{\vh}{\mathbf{h}}
\newcommand{\vv}{\mathbf{v}}
\newcommand{\vu}{\mathbf{u}}
\newcommand{\ve}{\mathbf{e}}

\newcommand{\calW}{\mathcal{W}}

\newcommand{\bs}{\backslash}
\newcommand{\bk}{\backslash}
\newcommand{\ds}{\displaystyle}
\newcommand{\op}{\text{op}}
\newcommand{\loc}{\text{loc}}

\def\Xint#1{\mathchoice
   {\XXint\displaystyle\textstyle{#1}}%
   {\XXint\textstyle\scriptstyle{#1}}%
   {\XXint\scriptstyle\scriptscriptstyle{#1}}%
   {\XXint\scriptscriptstyle\scriptscriptstyle{#1}}%
   \!\int}
\def\XXint#1#2#3{{\setbox0=\hbox{$#1{#2#3}{\int}$}
     \vcenter{\hbox{$#2#3$}}\kern-.5\wd0}}

\def\dashint{\Xint-}

\newcommand{\pp}{{p(\cdot)}}
\newcommand{\Pp}{\mathcal{P}}
\newcommand{\Lpp}{L^\pp}
\newcommand{\cpp}{{p'(\cdot)}}
\newcommand{\Lcpp}{L^\cpp}
\newcommand{\rr}{{r(\cdot)}}

\newcommand{\qq}{{q(\cdot)}}
\newcommand{\qp}{\qq}

\newcommand{\up}{{u(\cdot)}}
\newcommand{\ucp}{{u'(\cdot)}}
\newcommand{\vp}{{v(\cdot)}}

\newcommand{\rp}{{r(\cdot)}}

\DeclareMathOperator{\ind}{\chi}
\newcommand{\dd}{\hspace{2pt}\mathrm{d}}

\newcommand{\mc}{\mathcal}

\newcommand{\calA}{\mathcal{A}}

\begin{document}

\title{Matrix-weighted bounds in variable Lebesgue spaces}

\author{Zoe Nieraeth}
\thanks{Z. N. is supported by the Basque government through project GV IT1615-22}
\address{Zoe Nieraeth (she/her), University of the Basque country (UPV/EHU), Leioa, Spain}
\email{zoe.nieraeth@gmail.com}

\author{Michael Penrod}
\address{Michael Penrod, University of Alabama, Tuscaloosa Alabama}
\email{mjpenrod@crimson.ua.edu}
\thanks{M.P. thanks David Cruz-Uribe for his helpful discussions and guidance.}




\begin{abstract}
In this paper we prove boundedness of Calder\'on-Zygmund operators and the Christ-Goldberg maximal operator in the matrix-weighted variable Lebesgue spaces recently introduced by Cruz-Uribe and the second author. Our main tool to prove these bounds is through bounding a Goldberg auxiliary maximal operator.

As an application, we obtain a quantitative extrapolation theorem for matrix-weighted variable Lebesgue spaces from the recent framework of directional Banach function spaces of the first author.
\end{abstract}
\maketitle
\section{Introduction}
In this paper, we prove that the Christ-Goldberg maximal operator and Calder\'{o}n Zygmund operators are bounded on matrix-weighted variable Lebesgue spaces. In doing so, we further develop the theories of variable Lebesgue spaces and matrix weights. Before stating our results, we provide a brief summary of the history of variable Lebesgue spaces, matrix weights, and the Christ-Goldberg maximal operator to motivate our results. We defer the necessary definitions and notations to Section~\ref{sec:Prelim}. 

To merge the theory of matrix weights and the Christ-Goldberg maximal operator with variable Lebesgue spaces, we begin with the theory of scalar $A_p$ weights. The study of $A_p$ weights dates back to the early 1970's when Muckenhoupt \cite{MR0293384} proved the Hardy-Littlewood maximal operator is bounded on $L^p(v)$ if and only if $v\in A_p$. Given $1<p<\infty$, a weight $v$ is a scalar $A_p$ weight if 
\[ [v]_{A_p} : = \sup_Q \dashint_Q v(x) \dd x \left( \dashint_Q v(y)^{-\frac{p'}{p}}\dd y \right)^{\frac{p}{p'}} <\infty,\]
where the supremum is taken over all cubes $Q\subset \R^n$. To generalize $A_p$ weights to the variable exponent setting, we need to define an alternative version of $A_p$, which we denote by $\calA_p$. Given $1<p<\infty$, a weight $w$ is a scalar $\calA_p$ weight if 
\[ [w]_{\calA_p} : = \sup_Q \left( \dashint_Q w(x)^p\dd x \right)^{\frac{1}{p}} \left( \dashint_Q w(y)^{-p'}\dd y \right)^{\frac{1}{p'}} <\infty.\]
Notice that $w\in \calA_p$ if and only if $v=w^p \in A_p$ with $[w]_{\calA_p} = [v]_{A_p}^{\frac{1}{p}}$. The definition of classical $A_p$ weights is based on viewing the weight $v$ as a measure in the $L^p(v)$ norm, i.e., defining
\[\| f\|_{L^p(v)} = \left( \int_{\R^n} |f(x)|^p v(x)\dd x \right)^{\frac{1}{p}}.\]
The definition of $\calA_p$ is based on viewing the weight $w=v^{\frac{1}{p}}$ as a multiplier, in which case, we define $L^p_w(\R^n)$ as
\[ \| f\|_{L^p_w(\R^n)} = \| wf\|_{L^p(\R^n)} = \left(\int_{\R^n} |w(x) f(x)|^p \dd x \right)^{\frac{1}{p}}.\]
This approach of using weights as multipliers was first adopted by Muckenhoupt and Wheeden \cite{MR340523} to define the ``off-diagonal" $A_{p,q}$ weights used with fractional intergral operators. 

We now summarize the history of variable Lebesgue spaces. The early development of variable Lebesgue spaces is due to Orlicz, Nakano, and Zhikov. See \cite{cruz-uribe_variable_2013} for a thorough history of their contributions. We focus on the modern history. In the 1990's, interest in variable Lebesgue spaces grew due to their applications to multiple kinds of engineering and modeling problems (see \cite{MR1905047,MR1930392,MR2155087,MR1810360,MR2099981}). Because of these applications, there was a need to extend the techniques and results of harmonic analysis to variable Lebesgue spaces. 

One important problem was determining what conditions on the exponent function $\pp$ ensure that the Hardy-Littlewood maximal operator is bounded on the variable Lebesgue space $\Lpp(\R^n)$. The first major result about this problem is due to Diening \cite{diening_maximal_2004}, where he showed it is sufficient to assume $\pp$ is constant outside some large ball and satisfies the local log-H\"{o}lder continuity condition (see inequality \eqref{LH0} below). Diening's result was generalized by Cruz-Uribe, Fiorenza, and Neugebauer \cite{MR1976842,MR2041952}, where they replaced the requirement that $\pp$ be constant outside a ball with the assumption that $\pp$ is log-H\"{o}lder continuous at infinity (see inequality \eqref{LHinfty} below). Later, Diening \cite{MR2166733} gave necessary and sufficient conditions for boundedness of the Hardy-Littlewood maximal operator that are hard to check, but useful for theoretical purposes.
The definition of $\calA_p$ weights leads to a natural generalization in the variable exponent setting. We can rewrite the definition of $[w]_{\calA_p}$ using $L^p$ norms, i.e.,
\[ [w]_{\calA_p} = \sup_Q |Q|^{-1}\| w\ind_Q\|_{L^p(\R^n)}\| w^{-1}\ind_Q\|_{L^{p'}(\R^n)},\]
and then replace the $L^p$ norms with variable exponent $\Lpp$ norms. Given an exponent function $\pp$, a weight $w$ is a scalar $\calA_\pp$ weight if 
\[ [w]_{\calA_\pp} : = \sup_Q |Q|^{-1} \| w\ind_Q\|_{\Lpp(\R^n)} \| w^{-1}\ind_Q\|_{\Lcpp(\R^n)}<\infty.\]

Cruz-Uribe, Diening, and H\"ast\"o \cite{MR2837636} proved the Hardy-Littlewood maximal operator is bounded on the weighted variable Lebesgue space $\Lpp_w(\R^n)$, when $w\in \calA_\pp$. Cruz-Uribe, Fiorenza, and Neugebauer \cite{MR2927495} proved weighted strong and weak-type norm inequalities for the Hardy-Littlewood maximal operators on $\Lpp_w(\R^n)$ when $\pp$ is log-H\"{o}lder continuous and $w\in \calA_\pp$. Cruz-Uribe and Cummings \cite{MR4387458} extended this work to prove weighted norm inequalities for the Hardy-Littlewood maximal operator on $\Lpp_w(\R^n)$ over spaces of homogeneous type.

Cruz-Uribe, Fiorenza, Martell, and Perez \cite{MR2210118} showed many classical operators in harmonic analysis are bounded on $\Lpp(\R^n)$ whenever the Hardy-Littlewood maximal operator is bounded on $\Lpp(\R^n)$. They did so by applying the theory of Rubio de Francia extrapolation and weighted norm inequalities to variable Lebesgue spaces. Cruz-Uribe and Wang \cite{MR3572271} extended the theory of Rubio de Francia extrapolation to weighted variable Lebesgue spaces $\Lpp_w(\R^n)$, and applied this theory to prove weighted norm inequalities for several classical operators on these spaces. One underlying assumption in their results was $\pp$ and $w$ are such that the Hardy-Littlewood maximal operator is bounded on $\Lpp_w(\R^n)$. 

The study of matrix weights began with Nazarov, Treil, and Volberg in the 1990's. Motivated by applications to Toeplitz operators and stationary processes, Treil and Volberg \cite{treil_wavelets_1997}  developed the matrix $A_2$ condition to prove bounds for the Hilbert transform on matrix weighted $L^2$. Nazarov and Treil \cite{nazarov_hunt_1996} defined a matrix $A_p$ condition using a more complicated approach involving norm functions. They used Bellman function techniques to prove their $A_p$ condition is necessary and sufficient for the boundedness of the Hilbert transform on matrix weighted $L^p$ spaces. Their work generalized the scalar results of Hunt, Muckenhoupt, and Wheeden in \cite{hunt_weighted_1973}.

In \cite{christ_vector_2001}, Christ and Goldberg introduced the Christ-Goldberg maximal operator to study singular integral operators on matrix weighted $L^2$. Matrix weights are symmetric, positive definite matrix functions. Given a matrix weight $V$, Christ and Goldberg defined $M_V$ by 
\[ M_V \vf(x) = \sup_Q \dashint_Q | V^{\frac{1}{2}}(x) V^{-\frac{1}{2}}(y) \vf(y)|\dd y \, \ind_Q(x), \]
and they proved if $V\in A_2$, then there exists $\delta >0$ such that $M_V$ is bounded on $L^p$ when $|p-2|<\delta$. 

Goldberg \cite{goldberg_matrix_2003} generalized this work to matrix weighted $L^p$ for $1 < p <\infty$. He extended the definition of matrix $A_2$ and the Christ-Goldberg maximal operator to $p\neq 2$. Given a weight $V$, Goldberg defined $M_{V,p}$ by
\begin{align}\label{eqn:CGMaxOpOnLp}
 M_{V,p} \vf(x)= \sup_Q \dashint_Q |V^{\frac{1}{p}}(x) V^{-\frac{1}{p}}(y) \vf(y)|\dd y \, \ind_Q(x),
\end{align}
for $\vf \in L^1_{\loc}(\R^n;\C^d)$. Goldberg proved that if $V \in A_p$, then there exists $\delta>0$ such that $M_{V,p}$ is bounded from $L^q(\R^n;\C^d)$ to $L^q(\R^n)$ whenever $|p-q|<\delta$. 

Roundeko \cite{roudenko_matrix-weighted_2002} developed some aspects of Littlewood-Paley function space theory by introducing matrix weighted Besov spaces and the following definition of matrix $A_p$: given $1 < p <\infty$, $V$ is a matrix $A_p$ weight if 
\[ [V]_{A_p} :=\sup_Q  \dashint_Q \left( \dashint_Q |V^{\frac{1}{p}}(x) V^{-\frac{1}{p}}(y)|_{\op}^{p'} \dd y \right)^{\frac{p}{p'}} \dd x <\infty.\]
This definition is equivalent to the original definition given by Nazarov, Treil, and Volberg, and also reduces to the classical definition of $A_p$ weights as measures when the matrix weight is $1\times 1$. As in the scalar case, we define an alternative class of matrix weights that generalizes the view of weights as multipliers. Given $1 < p <\infty$, a matrix weight $W$ is a matrix $\calA_p$ weight if 
\[ [W]_{\calA_p} : = \sup_Q \left( \dashint_Q \left( \dashint_Q | W(x)W^{-1}(y)|_{\op}^{p'}\dd y \right)^{\frac{p}{p'}} \dd x \right)^{\frac{1}{p}} <\infty.\]
When $W$ is $1\times 1$, this definition reduces to the scalar $\calA_p$ definition. Moreover, the relationship between $A_p$ and $\calA_p$ is the same as in the scalar case: $W \in \calA_p$ if and only if $V=W^p \in A_p$ with $[W]_{\calA_p} = [V]_{A_p}^{\frac{1}{p}}$. 

As we did in the scalar case, we rewrite the definition of $[W]_{\calA_p}$ using norms to get
\[ [W]_{\calA_p} = \sup_Q |Q|^{-1} \big\| \big\| |W(x) W^{-1}(y)|_{\op} \ind_Q(y)\big\|_{L^{p'}_y(\R^n)}\ind_Q(x)\big\|_{L^p_x(\R^n)}. \]
By replacing the $L^p$ norms with $\Lpp$ norms we get the definition of matrix $\calA_\pp$. Given an exponent function $\pp$, a matrix weight $W$ is a matrix $\calA_\pp$ weight if 
\[ [W]_{\calA_\pp} : = \sup_Q |Q|^{-1} \big\| \big\| |W(x) W^{-1}(y)|_{\op} \ind_Q(y)\big\|_{\Lcpp_y(\R^n)} \ind_Q(x) \big\|_{\Lpp_x(\R^n)} <\infty.\]
This definition was introduced in \cite{ConvOpsOnVLS}, and is closely connected to the boundedness of averaging operators on $\Lpp(W)$.
With this alternative definition of matrix $\calA_\pp$ weights, we redefine the Christ-Goldberg maximal operator by 
\[ M_W \vf(x) = \sup_Q \dashint_Q |W(x) W^{-1}(y)\vf(y)| \dd y \, \ind_Q(x)\]
for $\vf \in L^1_{\loc}(\R^n;\F^d)$, where $\F$ is either the field $\R$ or $\C$.

In the constant exponent case, it was shown by Hunt, Muckenhoupt, and Wheeden in \cite{hunt_weighted_1973} that the $\calA_p$ condition of a scalar weight is equivalent to the boundedness in $L^p(w)$ of all Calder\'on-Zygmund operators. Setting $\Delta:=\{(x,x)\in\R^n\times\R^n:x\in\R^n\}$, let
\[
K:\R^n\times\R^n\backslash\Delta\to\C
\]
be a measurable mapping. We say that $K$ is a \emph{Dini Calder\'{o}n-Zygmund kernel} if we have the smoothness condition
\[
|K(x,y)-K(z,y)|\leq \omega\left(\frac{|x-z|}{|x-y|}\right)\frac{1}{|x-y|^n}
\]
for all $x,y,z$ satisfying $|x-y|\geq 2|x-z|>0$,
where $\omega:[0,1]\to[0,\infty)$ is an increasing, subadditive function satisfying $\omega(0)=0$ and the Dini condition
\[
\int_0^1\!\omega(t)\,\frac{\mathrm{d}t}{t}<\infty.
\]
We say that $T$ is a \emph{Calder\'{o}n-Zygmund operator} associated to the kernel $K$ if there is a $1<p<\infty$ for which $T$ is a bounded linear operator
\[
T:L^p(\R^n)\to L^p(\R^n),
\]
and if for all compactly supported functions $f\in L^p(\R^n)$ and a.e. $x\in\R^{n}\backslash\text{supp}(f)$ we have
\[
Tf(x)=\int_{\R^n}K(x,y)
    f(y)\dd y.
\]
For a Calder\'on-Zygmund operator $T$ and $\vf=(f^1,\ldots,f^d)\in L^p(\R^n;\F^d)$, we define $\widetilde{T}\vf:\R^n\to\F^d$ as
\[
\widetilde{T}\vf(x):=(Tf^1(x),\ldots,Tf^d(x)).
\]
The Hunt-Muckenhoupt-Wheeden theorem was extended to the matrix-weighted setting in the constant exponent case by Nazarov, Treil, and Volberg, see \cite{Tr89, nazarov_hunt_1996, TV97a, treil_wavelets_1997, Vo97}, who proved that the boundedness of $\widetilde{T}$ in $L^p_W(\R^n;\F^d)$ for all Calder\'on-Zygmund operators $T$ is equivalent to the condition $W\in\calA_p$.

Our main result is a variable exponent version of the Hunt-Muckenhoupt-Wheeden theorem for matrix weights.
\begin{TheoremLetter}\label{thm:A}
Let $p(\cdot)\in\Pp(\R^n)\cap LH(\R^n)$ with $1<p_-\leq p_+<\infty$ and let $W:\R^n\to S_d$ be a matrix weight satisfying $|W^{-1}|_{\text{op}}\in L^{p'(\cdot)}_{\text{loc}}(\R^n)$. Then the following are equivalent:
\begin{enumerate}[(i)]
    \item\label{it:thmA1} $\widetilde{T}:\Lpp_W(\R^n;\F^d)\to \Lpp_W(\R^n;\F^d)$ for all Calder\'on-Zygmund operators $T$;
    \item\label{it:thmA2} $M_W:L^{p(\cdot)}(\R^n;\F^d)\to L^{p(\cdot)}(\R^n)$;
    \item\label{it:thmA3} $W\in\calA_{\pp}$.
\end{enumerate}
\end{TheoremLetter}
For the lower bounds of the operators, we essentially adapt the work of Goldberg in \cite{goldberg_matrix_2003} to the variable exponent setting. For the upper bounds, we follow the strategy of Kakaroumpas and the first author in \cite{KN24} to bound the convex body sparse operator using a Goldberg auxiliary maximal operator. Parts of the theory in \cite{KN24} are proven in the general setting of directional Banach function spaces introduced by the first author in \cite{Ni24b}. As matrix-weighted variable Lebesgue are directional Banach function spaces (see \cite[Section~8.1]{Ni24b}), these results are directly applicable here. 

We emphasize that the main novelty and difficulty in the proof of Theorem~\ref{thm:A} is in the implications \ref{it:thmA3}$\Rightarrow$\ref{it:thmA1},\ref{it:thmA2}. Our strategy involves bounding the Goldberg auxiliary operator
\[
M'_{W,p(\cdot)}:L^{p(\cdot)}(\R^n;\F^d)\to L^{p(\cdot)}(\R^n),
\]
where
\[ 
M'_{W,\pp} \vf(x) = \sup_Q  \left(\dashint_Q |\calW_Q^\pp W(y)^{-1} \vf(y)| \dd y\right) \ind_Q(x).
\]
This operator was introduced by Goldberg in the constant exponent case in \cite{goldberg_matrix_2003}. Our result is as follows.

\begin{TheoremLetter}\label{thm:B}
Let $\pp\in\Pp(\R^n)\cap LH(\R^n)$ with $1<p_-\leq p_+<\infty$ and let $W\in\calA_\pp$. Then there is an increasing function $\phi:[1,\infty)\to(0,
\infty)$, depending only on $n$, $d$, $p_-$, $p_+$, $p_\infty$, and the $LH(\R^n)$ constants of $\pp$, such that for all $\vf\in L^{p(\cdot)}(\R^n;\F^d)$ we have
\[
\|M'_{W,\pp}\vf\|_{L^{p(\cdot)}(\R^n)}\leq\phi([W]_{\calA_\pp})\|\vf\|_{L^{p(\cdot)}(\R^n;\F^d)}.
\]
\end{TheoremLetter}
Goldberg's interpolation argument of \cite{goldberg_matrix_2003} cannot be adapted to the variable exponent setting. Thus, the proof of Theorem \ref{thm:B} is much harder and requires a completely different argument. We adapt the argument from Cruz-Uribe, Diening, and H\"ast\"o in \cite{MR2837636}, where they proved if $1/\pp \in LH(\R^n)$, then the Hardy-Littlewood maximal operator is bounded on $\Lpp_w(\R^n)$ if and only if $w\in \calA_\pp$. We also use the recently proven reverse H\"{o}lder inequality for $\calA_\pp$ weights of Cruz-Uribe and the second author in \cite{PenrodNormRH}.

As in \cite{KN24}, the boundedness of another, equivalent, auxiliary operator $M''_{W,\pp}$ (defined in Section \ref{sec:ProofOfA}) can be used to prove the boundedness of the convex body operator of Nazarov, Petermichl, Treil and Volberg in \cite{NPTV17}. In combination with \cite[Proposition~5.6]{Ni24b}, the boundedness of Calder\'on-Zygmund operators then follows from \cite[Corollary~D]{Ni24b}, and that of the Goldberg maximal operator from \cite[Theorem~C]{Ni24b} applied to the directional Banach function space $\mathbf{X}=\Lpp_W(\R^n;\F^d)$. To prove the implication \ref{it:thmA1}$\Rightarrow$\ref{it:thmA3}, we use a generalization of the non-degeneracy condition of Stein \cite{St93} which was used to prove this implication for general directional Banach function spaces in \cite[Theorem~6.16]{KN24}.

As a consequence of Theorem~\ref{thm:A} and \cite[Theorem~8.1]{Ni24b}, we obtain the following extrapolation theorem:
\begin{TheoremLetter}\label{thm:C}
Let $1\leq p_0\leq \infty$, let $V$ be a set, and let $S:V\to L^0(\R^n;\F^d)$. Suppose
\[
T:\bigcup_{W\in \calA_{p_0}}S^{-1}(L^{p_0}_W(\R^n;\F^d))\to L^0(\R^n;\F^d)
\]
is a map for which there is an increasing function $\phi:[1,\infty)\to(0,\infty)$ such that for all $W\in \calA_{p_0}$ and all $\vf\in V$ with $S\vf\in L^{p_0}_W(\R^n;\F^d)$ we have
\[
\|T\vf\|_{L^{p_0}_W(\R^n;\F^d)}\leq\phi([W]_{\calA_{p_0}})\|S\vf\|_{L^{p_0}_W(\R^n;\F^d)}.
\]
Let $\pp\in\Pp(\R^n)\cap LH(\R^n)$ with $1<p_-\leq p_+<\infty$ and let $W\in\calA_\pp$. Then $T\vf$ is well-defined for all $\vf\in V$ with $S\vf\in \Lpp_W(\R^n;\F^d)$, and there is an increasing function $\psi:[1,\infty)\to(0,\infty)$, depending only on $p_0$, $n$, $d$, $p_-$, $p_+$, $p_\infty$, and the $LH(\R^n)$ constants of $\pp$, for which
\[
\|T\vf\|_{\Lpp_W(\R^n;\F^d)}
\leq\psi([W]_{\calA_\pp})\|S\vf\|_{\Lpp_W(\R^n;\F^d)}.
\]
\end{TheoremLetter}
The remainder of this paper is organized as follows. In Section \ref{sec:Prelim}, we state the relevant definitions and lemmas about variable Lebesgue spaces. In Section \ref{sec:CalAppWeights}, we state the lemmas needed for scalar and matrix $\calA_\pp$ weights. In Section \ref{sec:MaxOps}, we prove the necessary lemmas about maximal operators. We also prove that $M'_{W,\pp} \vf$ is finite almost everywhere. This is a consequence of the fact that $M_\pp |\vf|$ is finite almost everywhere, a fact that seems well-known, but is not in the literature. In Section \ref{sec:MainProof}, we prove Theorems \ref{thm:A}, \ref{thm:B}, and \ref{thm:C}.

Throughout this paper, we will use the following notation. We use $n$ to denote the dimension of the Euclidean space $\R^n$, and $d$ will denote the dimension of matrix and vector-valued functions. We denote the coordinate basis in $\F^d$ by $\{\ve_i\}_{i=1}^d$. Given a vector-valued function $\vf: \R^n \to \F^d$, we denote its $i$th component $\vf\cdot \ve_i$ by $f^i$. When we use cubes $Q$, we assume their sides are parallel to the coordinate axes. Given two values $A$ and $B$, we will write $A \lesssim B$ if there exists a constant $c$ such that $A \leq cB$. We write $A \approx B$ if $A \lesssim B$ and $B \lesssim A$. We will often indicate the parameters constants depend on by writing, for example, $C(n,\pp)$ or $\lesssim_{n,\pp}$  By a (scalar) weight $w$ we mean a non-negative, locally integrable function such that $w(x)>0$ almost everywhere.
%


\section{Preliminaries}\label{sec:Prelim}

We begin with the basic definitions and lemmas about variable Lebesgue spaces. We refer the reader to \cite{cruz-uribe_variable_2013} and \cite{diening_lebesgue_2011} for a thorough treatment of the subject.

An exponent function is a Lebesgue measurable function $\pp: \R^n \to [1,\infty]$. Denote the collection of all exponent functions on $\R^n$ by $\Pp(\R^n)$. Given a set $E\subseteq \R^n$, define
\[p_+(E) = \esssup_{x\in E} p(x), \hskip.5cm \text{and} \hskip.5cm p_-(E) =\essinf_{x\in E} p(x).\]
Let $p_+=p_+(\R^n)$ and $p_- = p_-(\R^n)$. If $0 < |E|<\infty$, define the harmonic mean of $\pp$ on $E$, denoted $p_E$, by 
\[\frac{1}{p_E} = \dashint_E \frac{1}{p(x)} \dd x.\]
Define the conjugate exponent to $\pp$, denoted $\cpp$, by
\[\frac{1}{p(x)} + \frac{1}{p'(x)} = 1,\]
for all $x\in \R^n$, where we use the convention that $\frac{1}{\infty} = 0$.

Given $\pp\in \Pp(\R^n)$, define the modular associated with $\pp$ by 
\[\rho_\pp(f) = \int_{\R^n\bk \Omega_\infty} |f(x)|^{p(x)} dx + \|f\|_{L^\infty(\Omega_\infty)},\]
where $\Omega_\infty = \{x\in \R^n: p(x) = \infty\}$. Define $\Lpp(\R^n)$ to be the collection of Lebesgue measurable functions $f:\R^n \to \R$ such that
\[\|f\|_{\Lpp(\R^n)} := \inf\{\lambda>0: \rho_\pp(f/\lambda)\leq 1\}<\infty.\]
If $f$ depends on two variables, $x$ and $y$, we specify which variable the norm is taken with respect to with subscripts, e.g., $\Lpp_x$ and $\Lpp_y$.

A weight is a non-negative function $w\in L^1_{\loc}(\R^n)$ with $0<w(x)<\infty$ for almost every $x$. Given a scalar weight $w$, define $\Lpp_w(\R^n)$ to be the weighted space with norm $\|f\|_{\Lpp_w(\R^n)} = \|wf\|_{\Lpp(\R^n)}$. 

We now state some important lemmas. The first lemma relates the norm to the modular.
\begin{lemma}\cite[Proposition 2.12]{cruz-uribe_variable_2013}\label{ModNormEquiv}
Given $\pp \in \Pp(\R^n)$ with $p_+<\infty$, $f\in \Lpp(\R^n)$ if and only if $\rho_\pp(f)<\infty$.
\end{lemma}
\begin{lemma}\cite[Proposition 2.21]{cruz-uribe_variable_2013}\label{prop:NormalizedMod}
Given $\pp \in \Pp(\R^n)$, if $f \in \Lpp(\R^n) $with $\| f\|_{\Lpp(\R^n)} >0$, then $\rho_\pp(f/\|f\|_{\Lpp(\R^n)} )\leq1$. Furthermore, $\rho_\pp(f/\|f\|_{\Lpp(\R^n)})=1$ for all non-trivial $f\in \Lpp(\R^n)$ if and only if $p_+(\R^n\bs\Omega_\infty)<\infty$. 
\end{lemma}
\begin{lemma}\cite[Corollary 2.23]{cruz-uribe_variable_2013}\label{cor:ModNormEquiv}
Let $\pp \in \Pp(\R^n)$ with $p_+<\infty$. If $\|f\|_{\Lpp(\R^n)}>1$, then 
\[\rho_\pp(f)^{\frac{1}{p_+}} \leq \| f\|_{\Lpp(\R^n)}\leq \rho_\pp(f)^{\frac{1}{p_-}}.\]
If $0 <\| f\|_{\Lpp(\R^n)} \leq 1$, then
\[\rho_\pp(f)^{\frac{1}{p_-}} \leq \| f\|_{\Lpp(\R^n)} \leq \rho_\pp(f)^{\frac{1}{p_+}}.\]
\end{lemma}

H\"{o}lder's inequality is an important tool for proving inequalities in $L^p$ spaces. An analogous version of this inequality holds in variable Lebesgue spaces. 
\begin{lemma}\label{Holder}
Given $\pp \in \Pp(\R^n)$, for all $f \in \Lpp(\R^n)$ and $g \in \Lcpp(\R^n)$ we have $fg \in L^1(\R^n)$ and 
\[\int_{\R^n} |f(x) g(x) |dx \leq K_\pp \|f\|_{\Lpp(\R^n)}\|g\|_{\Lcpp(\R^n)}.\]
where $K_\pp\leq 4$ is a constant depending only on $\pp$. If $1 < p_- \leq p_+<\infty$, then $K_\pp \leq 2$. 
\end{lemma}

Next, we provide the definition of log-H\"{o}lder continuity, which plays an important role in many results involving variable Lebesgue spaces. 
\begin{definition}\label{LH:def}
A function $\rr: \R^n \to \R$ is locally log-H\"{o}lder continuous, denoted by $\rr \in LH_0(\R^n)$, if there exists a constant $C_0$ such that for all $x,y\in \R^n$, $|x-y|<1/2$,
\begin{align}\label{LH0}
|r(x)-r(y)| \leq \frac{C_0}{-\log(|x-y|)}.
\end{align}
We say that $\rr$ is log-H\"{o}lder continuous at infinity, denoted $\rr \in LH_\infty(\R^n)$, if there exist constants $C_\infty$ and $r_\infty$ such that for all $x\in \R^n$,
\begin{align}\label{LHinfty}
|r(x)-r_\infty| \leq \frac{C_\infty}{\log(e+|x|)}.
\end{align}
If $\rr$ is log-H\"{o}lder continuous locally and at infinity, we will denote this by writing $\rr\in LH(\R^n)$. 
\end{definition}
\begin{remark}\label{rem:ReciprocalsInLH}
In the literature, some authors state results with the assumption that $1/\pp\in LH(\R^n)$, and others assume $\pp \in LH(\R^n)$. By \cite[Proposition 2.3]{cruz-uribe_variable_2013}, if $p_+<\infty$, then $\pp\in LH(\R^n)$ if and only if $1/\pp \in LH(\R^n)$. When using the log-H\"{o}lder constants for functions other than $\pp$, e.g. $1/\pp$, we will use the notation $C_0(1/\pp)$ and $C_\infty(1/\pp)$.
\end{remark}

In our proofs, we need to compare powers of different sized cubes. We do so with the following lemmas.
\begin{lemma}\cite[Lemma 3.24]{cruz-uribe_variable_2013},\cite[Lemma 4.1.6]{diening_lebesgue_2011}\label{lem:DieningCondition}
Let $\rp:\R^n\to  [0,\infty)$ with $r_+<\infty$. If $\rp \in LH_0(\R^n)$, then there exists a constant $C_D$ such that for all cubes $Q\subset \R^n$, 
\[|Q|^{r_-(Q)-r_+(Q)} \leq C_D.\]
In fact, we may take $C_D= \max\{(2\sqrt{n})^{n(r_+-r_-)}, \exp(C_0 (1+\log_2\sqrt{n}))\}$.
\end{lemma}
\begin{remark}
    The constant $C_D$ above is found by tracking the constants in \cite[Lemma 3.24]{cruz-uribe_variable_2013}. Since $C_D$ depends on the function $\rp$, when working with multiple functions, we specify the dependence by $C_D(\rp)$, for example.
\end{remark}

We need the following lemma to track the constants in Lemma \ref{lem:LHImpliesK_0} below. 
\begin{lemma}\label{lem:LHInftyRemainderIneq}
Let $\up: \R^n\to [0,\infty)$ be such that $\up \in LH_\infty(\R^n)$ and $0 < u_\infty <\infty$, and for $t>0$, let $R_t(x)=(e+|x|)^{-nt}$. Then for any set $E$ with $|E|<\infty$, and any function $F$ with $0 \leq F(y)\leq 1$ for $y\in E$, 
\begin{align*}
\int_E F(y)^{u(y)} \dd x \leq e^{ntC_\infty} \int_E F(y)^{u_\infty} \dd x + \int_E R_t(y)^{u_-} \dd x, 
\end{align*}
and
\begin{align*}
\int_E F(y)^{u_\infty} \dd x \leq e^{ntC_\infty}\int_E F(y)^{u(y)} \dd x + \int_E R_t(y)^{u_-} \dd x.
\end{align*}
\end{lemma}
The following results are stated in terms of the $\calA_\pp$ condition, given in Definition~\ref{App}. We state the special case needed here. Given $\pp \in \Pp(\R^n)$, we say that $1 \in \calA_\pp$ if 
\[ [1]_{\calA_\pp} = \sup_Q |Q|^{-1} \| \ind_Q\|_{\Lpp(\R^n)} \| \ind_Q\|_{\Lcpp(\R^n)}<\infty.\]
\begin{remark}
If $\pp \in \Pp(\R^n)\cap LH(\R^n)$ with $ p_+<\infty$, then $1 \in \calA_\pp$. However, this condition is strictly weaker than log-H\"{o}lder continuity. (See \cite[Proposition 4.57, Example 4.59]{cruz-uribe_variable_2013}.) Note that this condition is referred to there as the $K_0$ condition.
\end{remark}

\begin{lemma}\cite[Lemma 4.5.3]{diening_lebesgue_2011}\label{CharFunctionNormIneq}
Given $\pp\in LH(\R^n)$ with $p_+<\infty$, for any cube $Q\subset \R^n$, 
\[\frac{1}{6}|Q|^{1/p_Q} \leq \| \ind_Q\|_{\Lpp(\R^n)} \leq 4K_\pp [1]_{\calA_\pp} |Q|^{1/p_Q}.\]
\end{lemma}
\begin{remark}
The hypotheses of Lemma \ref{CharFunctionNormIneq} can be relaxed by replacing $LH(\R^n)$ with the weaker $K_0$ condition. See \cite[Proposition 3.8]{TroyThesis} for details. 
\end{remark}
\begin{lemma}\label{CubeComparison}
Given $\pp \in \Pp(\R^n)\cap LH(\R^n)$ with $p_+<\infty$ and two cubes $Q_1, Q_2\subset \R^n$ with $Q_1 \subset Q_2$ and $|Q_2|\leq C|Q_1|$, we have
\[ |Q_1|^{-1/p_{Q_1}} \leq 24K_\pp [1]_{\calA_\pp} C |Q_2|^{-1/p_{Q_2}}.\]
\end{lemma}
\begin{proof}
Fix $Q_1, Q_2 \subset \R^n$ with $|Q_2| \leq C|Q_1|$. Since $\pp \in LH(\R^n)$, we have $\cpp \in LH(\R^n)$. Thus, by Lemma \ref{CharFunctionNormIneq}, for any cube $Q$,
\[\frac{1}{6} |Q|^{1/p'_Q} \leq \|\ind_{Q} \|_{\Lcpp(\R^n)} \leq 4K_\pp [1]_{\calA_\pp} |Q|^{1/p'_Q}.\]
Thus,
\begin{align*}
|Q_1|^{-1/p_{Q_1}}&  = |Q_1|^{-1}|Q_1|^{1/p'_{Q_1}} \leq 6|Q_1|^{-1} \| \ind_{Q_1}\|_{\Lcpp(\R^n)}\\
	& \leq 6|Q_1|^{-1} \|\ind_{Q_2}\|_{\Lcpp(\R^n)} \\
	& \leq 6(4K_\pp [1]_{\calA_\pp}) |Q_1|^{-1} |Q_2|^{1/p'_{Q_2}} \\
	& \leq 24K_\pp [1]_{\calA_\pp} C|Q_2|^{-1}|Q_2|^{1/p'_{Q_2}} \\
	& = 24K_\pp [1]_{\calA_\pp} C |Q_2|^{-1/p_{Q_2}}.
\end{align*}
\end{proof}

In the proof of Theorem \ref{thm:B}, we will need the following lemma about scalar multiples and quotients of log-H\"{o}lder continuous functions. 
\begin{lemma}\label{lem:QuotientsLH}
Let $s\in \R$ and $\pp, \up \in \Pp(\R^n)\cap LH(\R^n)$ with $p_+<\infty$. Then $\ds s\pp, \frac{\up}{\pp}\in LH(\R^n)$. 
\end{lemma}
\begin{proof}
Fix $s\in \R$ and fix $x\in \R^n$. Then we immediately get
\[ |sp(x)-sp_\infty| \leq s \frac{C_\infty}{\log(e+|x|)},\]
for all $x\in \R^n$. Likewise, 
\[ |sp(x)-sp(y)|\leq s \frac{C_0}{-\log|x-y|}\]
for all $x,y \in \R^n$ with $|x-y|<\frac{1}{2}$. Hence $s\pp \in LH(\R^n)$. 

Since $p_+<\infty$, by Remark \ref{rem:ReciprocalsInLH}, $1/\pp \in LH(\R^n)$. Thus,
\begin{align*}
 \left| \frac{u(x)}{p(x)} - \frac{u_\infty}{p_\infty} \right| & \leq \left| \frac{u(x)}{p(x)} - \frac{u_\infty}{p(x)} \right| + \left| \frac{u_\infty}{p(x)} - \frac{u_\infty}{p_\infty} \right|\\
 & \leq \frac{1}{p_-} | u(x)-u_\infty| + u_\infty \left| \frac{1}{p(x)}-\frac{1}{p_\infty}\right| \\
    &  \leq \frac{1}{p_-} \frac{C}{\log(e+|x|)} + u_\infty \frac{C}{\log(e+|x|)}\\
    & \leq \frac{C(\up, \pp)}{\log(e+|x|)}.
\end{align*}
Similarly, for all $x,y\in \R^n$ with $|x-y|<\frac{1}{2}$,
\begin{align*}
    \left| \frac{u(x)}{p(x)} - \frac{u(y)}{p(y)}\right| & \leq \left| \frac{u(x)}{p(x)} - \frac{u(y)}{p(x)} \right| + \left| \frac{u(y)}{p(x)}- \frac{u(y)}{p(y)}\right|\\ 
    & \leq \frac{1}{p_-} \left(\frac{C}{-\log|x-y|}\right)+ u_+ \left(\frac{C}{-\log|x-y|}\right)\\
    & \leq \frac{C(\up,\pp)}{-\log|x-y|}
\end{align*}

Thus, $\ds \frac{\up}{\pp} \in LH(\R^n)$.
\end{proof}

\section{\texorpdfstring{$\calA_\pp$}{} weights}\label{sec:CalAppWeights}
The properties and fine structure of both scalar and matrix $\calA_\pp$ weights have been developed recently. We collect the relevant results in this section. In \cite{PenrodNormRH}, the authors  proved a reverse H\"{o}lder inequality for scalar $\calA_\pp$ weights. 
\begin{theorem}\cite[Theorem 1.1 and Corollary 4.3]{PenrodNormRH}\label{thm:NormRH}
    Let $\pp \in \Pp(\R^n)\cap LH(\R^n)$ with $p_+<\infty$ and let $w$ be a scalar $\calA_\pp$ weight. Then there exist constants $C_\pp$ and $r>1$ such that for all cubes $Q\subset \R^n$, 
    \begin{align}\label{ineq:NormRH}
    |Q|^{-\frac{1}{rp_Q}} \| w\ind_Q \|_{L^{r\pp}(\R^n)} \leq C_\pp |Q|^{-\frac{1}{p_Q}} \| w\ind_Q\|_{\Lpp(\R^n)}.
    \end{align}
    In particular, we may take 
        \[ C_\pp = C^* [w]_{\calA_\pp}^{\frac{C_\infty p_+}{p_-^2}(p_++1)},\]
    where $C^*$ depends on $n, p_-,p_+$ and the log-H\"{o}lder constants of $\pp$, and
    \[ r = 1 + \frac{1}{C_* [w]_{\calA_\pp}^{\left(1+2\frac{C_\infty p_+}{p_\infty p_-}\right)p_+}},\]
    where $C_* = C(n,\pp, C_\infty)$. Moreover, if $s \in (1,r)$, then 

    \begin{align}\label{ineq:NormRH-s}
        |Q|^{-\frac{1}{sp_Q}} \|w\ind_Q\|_{L^{s\pp}(\R^n)} \leq 32 [1]_{\calA_\vp} C_\pp |Q|^{-\frac{1}{p_Q}} \| w\ind_Q\|_{\Lpp(\R^n)}
    \end{align}
    for all cubes $Q$, where $\vp$ is defined by
    \[ \frac{1}{s\pp} = \frac{1}{r\pp} + \frac{1}{\vp}.\]
\end{theorem}

Notice that in inequality \eqref{ineq:NormRH-s}, the constant $[1]_{\calA_{\vp}}$ implicitly depends on $s,r,$ and $\pp$. The following lemmas allow us to remove the dependence on $s$ and $r$, which is important for the proof of Theorem \ref{thm:B}. 

The first lemma combines multiple results to obtain quantitative estimates for $[1]_{\calA_\pp}$.
\begin{lemma}\label{lem:LHImpliesK_0}
Let $\pp \in \Pp(\R^n)$ with $1/\pp \in LH(\R^n)$. Then $1\in \calA_\pp$. In fact,  
\[ [1]_{\calA_\pp}\leq 8 C_D(1/\pp).\]
\end{lemma}
\begin{proof}
  Fix $\pp\in \Pp(\R^n)$ with $1/\pp\in LH(\R^n)$. By \cite[Theorem 4.4.8]{diening_lebesgue_2011}, if $1/\pp \in LH(\R^n)$, then for all pairwise disjoint collections $\calQ$ of cubes, 
  \[ \| A_\calQ f\|_{\Lpp(\R^n)}\leq 2 C_D(1/\pp) \| f\|_{\Lpp(\R^n)},\]
  where $A_\calQ f= \sum_{Q\in \calQ} \dashint_Q f(y)\, dy \chi_Q$. By \cite[Theorem 4.1]{ConvOpsOnVLS} applied to the scalar weight $1$, 
  \[ [1]_{\calA_\pp} \leq 4 \sup_Q \| A_Q\|_{\Lpp(\R^n)\to \Lpp(\R^n)},\]
  where $A_Qf = \dashint_Q f(y)\, dy \chi_Q$. Thus, 
  \[ [1]_{\calA_\pp} \leq 8 C_D(1/\pp).\]
\end{proof}

\begin{corollary}\label{cor:[1]_vpBound}
    Let $\pp \in \Pp(\R^n)$ with $1/\pp \in LH(\R^n)$. Let $r>1$ and $s\in (1,r)$. Define $\vp$ by 
    \[ \frac{1}{s\pp} = \frac{1}{r\pp} + \frac{1}{\vp}.\]
    Then $[1]_{\calA_\vp} \leq 8 C_D(1/\pp)$.
\end{corollary}
\begin{proof}
    We will show
    \begin{align}
        C_0\left( \frac{1}{\vp}\right) \leq C_0\left( \frac{1}{\pp}\right),
    \end{align}
    since this implies $C_D(1/\vp) \leq C_D(1/\pp)$. Observe that 
    \[ \left( \frac{1}{\vp}\right)_+ - \left( \frac{1}{\vp}\right)_-\leq 1 - 0 = 1.\]
    Also, $C_0\left( \frac{1}{\vp}\right) \leq C_0\left( \frac{1}{\pp}\right)$. To see this, observe that for all $x,y\in \R^n$ with $|x-y|<\frac{1}{2}$, we have
    \[ \left| \frac{1}{v(x)} - \frac{1}{v(y)}\right| = \left( \frac{1}{s}-\frac{1}{r}\right) \left| \frac{1}{p(x)}-\frac{1}{p(y)}\right| \leq \left( \frac{1}{s}-\frac{1}{r}\right) \frac{C_0(1/\pp)}{-\log|x-y|}.\]
    Thus, we have $C_0\left( \frac{1}{\vp}\right) \leq \left( \frac{1}{s}-\frac{1}{r}\right) C_0\left( \frac{1}{\pp}\right)$. Since $s\in (1,r)$, we have 
    \[ \frac{1}{s}-\frac{1}{r}\leq 1 - \frac{1}{r} =\frac{1}{r'}<1.\]
    Hence, $C_0\left( \frac{1}{\vp}\right) \leq C_0\left( \frac{1}{\pp}\right)$. 
    \end{proof}

Next, we define matrix weights and provide some important lemmas about them. Recall that the operator norm of a matrix $W$ is given by 
\[ |W|_{\op} =  \sup_{\substack{\vu\in \F^d\\ |\vu|=1}} |W \vu|.\]
Let $S_d$ denote the collection of $d\times d$ matrices with entries in $\F$ that are Hermitian and positive definite. Then a matrix weight is a measurable mapping $W:\R^n\to S_d$ such that $|W|_{\op}$ is a locally integrable function. A matrix weight is invertible as it is positive definite almost everywhere. 

The following lemmas play a useful role when working with vectors and matrices. 
\begin{lemma}{\cite[Lemma 3.2]{roudenko_matrix-weighted_2002}}\label{opNorm:equiv}
If $\{\ve_1, \ldots, \ve_d\}$ is any orthonormal basis in $\F^d$, then for any $d\times d$ matrix $V$, we have
\begin{align*}
	\frac{1}{d} \sum_{i=1}^d |V \ve_i| \leq |V|_{\op} \leq \sum_{i=1}^d |V \ve_i|.
\end{align*}
\end{lemma}
\begin{lemma}\label{SelfAdjointCommutes}
Let $U$ and $V$ be Hermitian $d\times d$ matrices. Then $|UV|_{\op} = |VU|_{\op}$.
\end{lemma}

We now define matrix weighted, variable Lebesgue spaces. These spaces were introduced in \cite{ConvOpsOnVLS}.
\begin{definition}
Define $L^0(\R^n;\F^d)$ to be the collection of vector-valued, Lebesgue measurable functions $\vf: \R^n\to \F^d$. Given $\pp \in \Pp(\R^n)$, define $\Lpp(\R^n;\F^d)$ to be the collection of functions $\vf\in L^0(\R^n;\F^d)$ such that
\[  \| \vf\|_{\Lpp(\R^n;\F^d)} =\||\vf|\|_{\Lpp(\R^n)}< \infty.\]
Given a matrix weight $W:\R^n\to S_d$, define $\Lpp_W(\R^n;\F^d)$ to be the collection of Lebesgue measurable functions $\vf : \R^n \to \F^d$ such that
\[ \| \vf\|_{\Lpp_W(\R^n;\F^d)} : = \| W\vf\|_{\Lpp(\R^n;\F^d)} <\infty.\]
\end{definition}

Before stating the relevant lemmas about matrix $\calA_\pp$ weights, we restate the definition for the reader's convenience.
\begin{definition}\label{App}
Given $\pp \in \Pp(\R^n)$ and a matrix weight $W:\R^n\to S_d$, we say that $W \in \calA_\pp$ if
\[[W]_{\calA_\pp} : = \sup_Q |Q|^{-1} \Big\| \big\| |W(x)W^{-1}(y)|_{op} \ind_Q(y)\big\|_{\Lcpp_y(\R^n)}\ind_Q(x)\Big\|_{\Lpp_x(\R^n)} < \infty.\]
\end{definition}

In \cite[Corollary 2.2]{goldberg_matrix_2003}, Goldberg proves that given any matrix weight $V \in A_p$, $[|V(\cdot)^{1/p}\vu|^p]_{A_p} \leq  [V]_{A_p}$ for all nonzero $\vu\in \C^d$. The same holds for matrix $\calA_\pp$ weights.
\begin{lemma}\cite[Lemma 5.10]{PenrodNormRH}\label{lem:Scalarization}
    Let $\pp \in \Pp(\R^n)$ and $W : \R^n\to S_d$ be a matrix weight. If $W \in \calA_\pp$, then for all nonzero $\vu\in \F^d$, $|W \vu|$ is a scalar $\calA_\pp$ weight with 
    \[ [|W\vu|]_{\calA_\pp} \leq 4 [W]_{\calA_\pp}.\]
\end{lemma}

The following proposition is extremely useful and plays a key role in the theory of matrix weights. 
\begin{prop}{\cite[Proposition 1.2]{goldberg_matrix_2003}\cite[Theorem 4.11]{bownik_extrapolation_2022}\cite[Proposition A.8]{MR4672185}}\label{thm:EllipsoidApprox}
Let $d<\infty$. Given any measurable Banach space norm function $r: \R^n \times \F^d \to [0,\infty)$, there exists a positive-definite, measurable matrix function $W: \R^n \to S_d$ such that 
\[r(x,\vu) \leq |W(x) \vu|\leq \sqrt{d}r(x,\vu).\]
\end{prop}
\begin{remark}
Proposition \ref{thm:EllipsoidApprox} was first proved in \cite[Proposition 1.2]{goldberg_matrix_2003}. However, this paper did not consider the question of the measurability of $W$. This was justified by Bownik and Cruz-Uribe in \cite[Theorem 4.11]{bownik_extrapolation_2022} for real-valued norms and matrix functions. The extension of measurability to complex-valued norms and matrix functions was achieved in \cite[Proposition A.8]{MR4672185}. 
\end{remark}
In \cite{ConvOpsOnVLS}, the authors prove characterizations of matrix $\calA_\pp$ weights in terms of averaging operators and reducing operators. We define the reducing operators needed in our results and state the relevant characterization. Given a matrix weight $W:\R^n \to S_d$, define the norm function $r(\cdot,\cdot): \R^n\times \F^d \to [0,\infty)$ by $r(x,\vu) = |W(x) \vu|$. Given a cube $Q \subset \R^n$, define the norm $\langle r\rangle_{\pp,Q}: \F^d \to [0,\infty)$ by 
\[\langle r\rangle_{\pp,Q} (\vu) : = |Q|^{-1/p_Q} \| r(\cdot, \vu) \ind_Q(\cdot)\|_{\Lpp(\R^n)} = |Q|^{-1/p_Q} \| \vu\ind_Q\|_{\Lpp_W(\R^n;\F^d)}.\]

By Proposition \ref{thm:EllipsoidApprox}, there exists a positive-definite, self-adjoint, (constant) matrix $\calW_Q^\pp$ such that $\langle r\rangle_{\pp,Q} (\vu) \approx |\calW_Q^\pp \vu|$ for all $\vu\in\F^d$. We call $\calW_Q^\pp$ the reducing operator associated to $r$ on $Q$. 

Let $r^*$ be the dual norm to $r$, given by $r^*(x,\vu) = |W^{-1}(x) \vu|$. Define the norm $\langle r^* \rangle_{\cpp,Q} :\F^d \to [0,\infty)$ by 
\[\langle r^* \rangle_{\cpp,Q}(\vu) : = |Q|^{-1/p'_Q} \| r^*(\cdot, \vu) \ind_Q(\cdot)\|_{\Lcpp(\R^n)} = |Q|^{-1/p'_Q} \| \vu\ind_Q\|_{\Lcpp_W(\R^n;\F^d)}.\]

Using Proposition \ref{thm:EllipsoidApprox}, there is a positive-definite, self-adjoint, (constant) matrix $\overline{\calW}_Q^\cpp$ such that $\langle r^* \rangle_{\cpp,Q}(\vu) \approx |\overline{\calW}_Q^\cpp \vu|$ for all $\vu\in \F^d$. We call $\overline{\calW}_Q^\cpp$ the reducing operator associated to $r^*$ on $Q$. 

We now use reducing operators to state an equivalent characterization of $\calA_\pp$. 
\begin{lemma}\cite[Proposition 4.7]{ConvOpsOnVLS}\label{prop:ApdotReducingOps}
Let $\pp \in \Pp(\R^n)$ and $W:\R^n \to S_d$ be a matrix weight. Then $W \in \calA_\pp$ if and only if
\begin{align*}
[W]_{\calA_\pp}^R = \sup_{Q} |\calW_Q^\pp \overline{\calW}_Q^\cpp|_{op} <\infty.
\end{align*}
Moreover, $[W]_{\calA_\pp}^R \approx [W]_{\calA_\pp}$, with implicit constants depending only on $d$.
\end{lemma}

As a consequence of Lemma \ref{prop:ApdotReducingOps} and the fact that $[W]_{\calA_\pp}^R = [W^{-1}]_{\calA_\cpp}^R$, we get the following corollary.
\begin{corollary}\label{cor:Symmetry}
    Given $\pp \in \Pp(\R^n)$ and a matrix weight $W$, $W \in \calA_\pp$ if and only if $W^{-1}\in \calA_{\cpp}$. Moreover, $[W]_{\calA_\pp} \approx [W^{-1}]_{\calA_\cpp}\approx [W]_{\calA_\pp}^R$ with implicit constants depending only on $d$.
\end{corollary}
%

%


\section{Maximal Operators}\label{sec:MaxOps}
Our proof of Theorem \ref{thm:B} relies on dyadic techniques. In this section, we define general dyadic grids, and state some of their properties.
\begin{definition}\label{DyadicGrid}
Let $\D$ be a collection of cubes in $\R^n$. $\D$ is a dyadic grid if it satisfies the following three properties:
\begin{enumerate}
\item given $Q \in \D$, $\ell(Q) = 2^{k}$ for some $k\in \Z$;

\item the subcollection $\D^k = \{Q\in \D: \ell(Q) = 2^k\}$ forms a partition of $\R^n$;

\item given any two cubes $Q,P \in \D$, we have $Q\cap P=\emptyset$, $Q \subseteq P$, or $P \subseteq Q$.
\end{enumerate}
Given a dyadic cube $Q$, define $\widehat{Q}$ to be the unique dyadic cube containing $Q$ with side length $2 \ell(Q)$.
\end{definition}

A well-known trick for bounding a maximal operator involves bounding it by a finite sum of maximal operators over dyadic grids. This relies on a special family of dyadic grids. 
\begin{definition}\label{t-DyadicGrid}
Given an $n$-tuple $t\in \{0,1/3\}^n$, define the dyadic grid $\D^t$ by
\[\D^t = \{2^k([0,1)^n + m + (-1)^k t): k\in \Z, \;m\in \Z^n\}.\]
\end{definition}

These special dyadic grids allow us to approximate any cube $Q \subset \R^n$ by a dyadic cube from $\D^t$ for some $t\in \{0,\frac{1}{3}\}^n$. This is known as the ``$1/3$" trick.
\begin{lemma}\label{1/3-trick}
Given any cube $Q \subset \R^n$, there exists $t\in \{0,1/3\}^n$ and $Q_t \in \D^t$ such that $Q \subseteq Q^t$ and $\ell(Q_t) \leq 6 \ell(Q)$.
\end{lemma}

Using this ``$1/3$" trick, we can bound $M'_{W,\pp}$ by a finite sum of corresponding dyadic maximal operators.
\begin{lemma}\label{MaxOpFiniteSumBound}
Given $\pp \in \Pp(\R^n)\cap LH(\R^n)$ with $p_+<\infty$ and a matrix weight $W: \R^n \to S_d$, for each $t \in \{0,1/3\}^n$, define $M_{W,\pp,\D^t}'$ by
\[M'_{W,\pp,\D^t} \vf(x) = \sup_{Q\in {\D^t}} \dashint_{Q} |\calW_Q^\pp W^{-1}(y) \vf(y)| dy \; \ind_Q(x),\]
for $\vf\in L^1_{loc}(\R^n;\F^d)$. Then for all $\vf\in L^1_{loc}(\R^n;\F^d)$ and $x\in \R^n$, 
\[ M'_{W,\pp} \vf(x) \leq C(n,d, \pp) \sum_{t\in \{0,1/3\}^n} M'_{W,\pp,\D^t} \vf(x).\]
\end{lemma}
\begin{proof}
Fix a cube $Q \subset \R^n$. Then by Lemma \ref{1/3-trick}, there exists $t\in \{0,1/3\}^n$ and $Q_t\in \D^t$ such that $Q\subseteq Q_t$ and $|Q_t|\leq 6^n |Q|$. Combining this with the definition of $\calW_Q^\pp$ and Lemma~\ref{CubeComparison}, we have for any $x\in Q$
\begin{align*}
&\dashint_Q |\calW_Q^\pp W^{-1}(y)\vf(y)| dy\, \ind_Q(x) \\
    &\qquad \leq \frac{6^n}{|Q_t|} \int_{Q_t} |\calW_Q^\pp W^{-1}(y) \vf(y)| \dd y\, \ind_{Q_t}(x) \\
    & \qquad \leq \frac{6^n\sqrt{d}}{|Q_t|} \int_{Q_t}  |Q|^{-\frac{1}{p_Q}} \| |WW^{-1}(y)\vf(y)|\ind_Q(y)\|_{\Lpp(\R^n)} \dd y\,  \ind_{Q_t}(x)\\
    &\qquad  \leq \frac{6^{2n} 24K_\pp [1]_{\calA_\pp}\sqrt{d}}{|Q_t|} \int_{Q_t} |Q_t|^{-\frac{1}{p_{Q_t}}} \| |WW^{-1}(y)\vf(y)| \ind_{Q_t}\|_{\Lpp(\R^n)} \dd y\,  \ind_{Q_t}(x)\\
    &\qquad  \leq 6^{2n}24K_\pp [1]_{\calA_\pp} \sqrt{d}\dashint_{Q_t} |\calW_{Q_t}^\pp W^{-1}(y) \vf(y)|\dd y\, \ind_{Q_t}(x)\\
    & \qquad \leq 6^{2n}24K_\pp [1]_{\calA_\pp} \sqrt{d} M'_{W,\pp, \D^t}\vf(x)\\   
    & \qquad \leq 6^{2n}24K_\pp [1]_{\calA_\pp} \sqrt{d} \sum_{t\in \{0,1/3\}^n}M'_{W,\pp, \D^t}\vf(x).
\end{align*}
Taking the supremum over all cubes $Q$ containing $x$, we get
\[M'_{W,\pp} \vf(x) \leq C(n,d,\pp)\sum_{t\in \{0,1/3\}^n} M'_{W,\pp,\D^t} \vf(x).\qedhere\]
\end{proof}

Define $L_c^\infty(\R^n;\F^d)$ to be the set of essentially bounded, compactly supported, vector-valued measurable functions $\vf : \R^n \to \F^d$. The monotone convergence property of $\Lpp(\R^n;\F^d)$ (see \cite{Ni24b} for the precise definition) allows us to work with bounded functions with compact support when proving Theorem~\ref{thm:B}.
\begin{lemma}\label{BddCompSuppReduction}
Let $\pp \in \Pp(\R^n)$ and let $W$ be an invertible matrix weight. Let $\vf\in \Lpp(\R^n;\F^d)$, and let $(\vf_k)_{k\geq 1}$ be a sequence in $L^\infty_c(\R^n;\F^d)$ for which for all $\vu\in \F^d$
\[
|\vf_k(y) \cdot \vu|\uparrow|\vf(y) \cdot \vu|
\]
as $k\to\infty$, for a.e. $y\in \R^n$. Then
\[
M'_{W,\pp} \vf_k(x)\uparrow M'_{W,\pp} \vf(x)
\]
as $k\to\infty$ for a.e. $x\in \R^n$.
\end{lemma}
\begin{remark}
By \cite[Proposition~2.4]{Ni24b}, the convergence $|\vf_k(x)\cdot \vu|\uparrow|\vf(x)\cdot \vu|$ for all $\vu\in\F^d$ for a.e. $x\in\R^n$ has several characterizations. A geometrically intuitive version can be given in terms of the convex-set valued mapping $\mc{K}(\vf)(x)$, defined as the smallest symmetric convex set containing the vector $\vf(x)\in\F^d$. Indeed, in this case the above convergence is equivalent to the assertion that $\mc{K}(\vf_k)\subseteq\mc{K}(\vf_{k+1})$ for all $k\geq 1$, and
\[
\bigcup_{k=1}^\infty\mc{K}(\vf_k)(x)=\mc{K}(\vf)(x)
\]
for a.e. $x\in\R^n$. The lemma states that this kind of increasing a.e. convex-set valued convergence of $(\vf_k)_{k\geq 1}$ is compatible with the increasing a.e. convergence of $(M'_{W,\pp}\vf_k)_{k\geq 1}$.
\end{remark}
\begin{proof}[Proof of Lemma~\ref{BddCompSuppReduction}]
By \cite[Proposition~2.4]{Ni24b}, the assumption on $(\vf_k)_{k\geq 1}$ is equivalent to the statement that we have $|A\vf_k(y)|\uparrow |A\vf(y)|$ for all $d\times d$ matrices $A$ for a.e. $y\in\R^n$. In particular, this means that for every cube $Q$ we have $|\calW_Q^\pp W^{-1}(y) \vf_k(y)|\uparrow |\calW_Q^\pp W^{-1}(y) \vf(y)|$ for a.e. $y\in\R^n$. Thus, by the monotone convergence theorem,
\[
\dashint_Q\!|\calW_Q^\pp W^{-1}(y) \vf_k(y)| \dd y\uparrow\dashint_Q\!|\calW_Q^\pp W^{-1}(y) \vf(y)| \dd y=\sup_{k\geq 1}\dashint_Q\!|\calW_Q^\pp W^{-1}(y) \vf_k(y)| \dd y.
\]
Letting $x\in\R^n$ and taking a supremum over all cubes $Q$ containing $x$ we obtain
\[
M'_{W,\pp} \vf_k(x)\uparrow \sup_{k\geq 1}\sup_{Q\ni x}\dashint_Q\!|\calW_Q^\pp W^{-1}(y) \vf_k(y)| \dd y=M'_{W,\pp} \vf(x).
\]
The assertion follows.
\end{proof}

We now prove a Calderon-Zygmund decomposition for $M'_{W,\pp,\D}$. 
\begin{lemma}\label{CZdecomp:AuxMaxOp}
Let $\pp \in \Pp(\R^n)\cap LH(\R^n)$ with $1<p_-\leq p_+<\infty$, $W \in \calA_\pp$, $\D$ be a dyadic grid, and $\vf \in L_c^\infty(\R^n;\F^d)$. Given $\lambda>0$, there exists a (possibly empty) collection $\{Q_j^\lambda\}_j\subseteq \D$ of pairwise disjoint cubes such that 
\begin{align}\label{CZdecomp:AuxMaxOp:decomp}
\Omega_\lambda := \{x \in \R^n : M'_{W,\pp,\D}\vf(x)>\lambda \} = \bigcup_j Q_j^\lambda,
\end{align}
and for each $Q_j^\lambda$, 
\begin{align}\label{CZdecomp:AuxMaxOp:ineq}
\lambda < \dashint_{Q_j^\lambda} |\calW_{Q_j^\lambda}^{\pp}W^{-1}(y) \vf(y)| dy \leq   24\cdot 4^n K_\pp [1]_{\calA_\pp} \sqrt{d} \lambda.
\end{align}

\end{lemma}
\begin{proof}
 Let $\D$ be a dyadic grid and fix $\lambda>0$. Let $\vf \in L_c^\infty(\R^n;\F^d)$. For every $x\in \Omega_\lambda$, there exists a maximal cube $Q_x\in \D$ with $x\in Q_x$ satisfying \eqref{CZdecomp:AuxMaxOp:ineq}. To see why maximality holds, observe that for any cube $Q$, by the definition of $\calW_Q^\pp$, Lemma \ref{Holder}, and Lemma \ref{CharFunctionNormIneq},
\begin{align*}
&\dashint_Q |\calW_Q^\pp W^{-1}(y) \vf(y)|\dd y \\
	& \leq \sqrt{d} \dashint_Q |Q|^{-1/p_Q} \| |W(\cdot)W^{-1}(y) \vf(y)| \ind_Q(\cdot)\|_{\Lpp(\R^n)}\dd y\\
	& \leq \sqrt{d}\dashint_Q |Q|^{-1/p_Q} |\vf(y)| \| |W(\cdot)W^{-1}(y)|_{\op} \ind_{Q}(\cdot)\|_{\Lpp(\R^n)} \dd y\\
	& \leq \sqrt{d}|Q|^{-1} |Q|^{-1/p_Q} \| \vf\|_{\Lpp(\R^n;\F^d)} \| \| |W(\cdot)W^{-1}(y) |_{\op}\ind_Q(\cdot)\|_{\Lpp(\R^n)} \ind_Q(y)\|_{\Lcpp_y(\R^n)}\\
	& \leq \sqrt{d}[W^{-1}]_{\calA_{\cpp}} |Q|^{-1/p_Q} \| \vf\|_{\Lpp(\R^n;\F^d)}.
\end{align*} 
Since $W\in\calA_\pp$, by Corollary \ref{cor:Symmetry}, $W^{-1}\in\calA_\cpp$, and so $[W^{-1}]_{\calA_\cpp}<\infty$.  Thus, the last expression converges to $0$ as $|Q|\to \infty$. Hence, there is a largest cube $Q_x$ containing $x$ satisfying \eqref{CZdecomp:AuxMaxOp:ineq}.  By the properties of dyadic grids, all cubes containing $x$ that satisfy \eqref{CZdecomp:AuxMaxOp:ineq} are contained in $Q_x$. This justifies the maximality of our choice of the cubes $Q_x$.

Clearly, $\Omega_\lambda \subseteq \bigcup_{x\in\Omega_\lambda} Q_x$.
To see the reverse inclusion, let $z\in Q_x$ for some $x\in \Omega_\lambda$. By our choice of cubes $Q_x$, 
\[M'_{W,\pp,\D}\vf(z) \geq \dashint_{Q_x} |\calW_{Q_x}^\pp W^{-1}(y) \vf(y)| dy >\lambda.\]
Hence $z \in \Omega_\lambda$. Since the $\D$ is countable, we may reindex $\{Q_x\}_{x\in \Omega_\lambda}$ by $\{Q_j\}_j$. By maximality and the properties of dyadic cubes (Definition \ref{DyadicGrid}), $\{Q_j\}_j$ is pairwise disjoint. This proves \eqref{CZdecomp:AuxMaxOp:decomp}.

The lower bound of \eqref{CZdecomp:AuxMaxOp:ineq} is immediate by our choice of $\{Q_j\}_j$. To see why the upper bound holds, observe that by the definition of the reducing operator $\calW_{Q_j}^\pp$, Lemma \ref{CubeComparison} applied to $Q_j$ and $\widehat{Q}_j$, and the maximality of the cubes $Q_j$ with respect to inequality \eqref{CZdecomp:AuxMaxOp:ineq}, we have
\begin{align*}
& \dashint_{Q_j} |\calW_{Q_j}^\pp W^{-1}(y) \vf(y)| dy \\
    & \qquad \leq 2^n \dashint_{\widehat{Q_j}} | \calW_{Q_j}^\pp W^{-1}(y) \vf(y)| dy\\
	& \qquad \leq 2^n \sqrt{d}\dashint_{\widehat{Q_j}} |Q_j|^{-1/p_{Q_j}} \| |W(\cdot) W^{-1}(y)\vf(y) | \ind_{Q_j}(\cdot)\|_{\Lpp(\R^n)}dy\\
    &	\qquad  \leq 2^n \sqrt{d} (24K_\pp [1]_{\calA_\pp})2^n \dashint_{\widehat{Q_j}} |\widehat{Q_j}|^{-1/p_{\widehat{Q_j}}}\| |W(\cdot) W^{-1}(y)\vf(y) | \ind_{\widehat{Q_j}}(\cdot)\|_{\Lpp(\R^n)}\\
    & \qquad \leq 2^n \sqrt{d} (24K_\pp [1]_{\calA_\pp})2^n \dashint_{\widehat{Q_j}} |\calW_{\widehat{Q_j}}^\pp W^{-1}(y) \vf(y)| dy\\
    &	\qquad  \leq  24\cdot4^n K_\pp [1]_{\calA_\pp} \sqrt{d} \lambda.
\end{align*}
This completes the proof.
\end{proof}

To prove Theorem \ref{thm:B}, we need a maximal operator based on norm averages. This maximal operator was utilized in \cite{MR2837636} to prove the Hardy-Littlewood maximal operator is bounded on weighted variable Lebesgue spaces. 
\begin{definition}\label{NormMaxOp:def}
Let $\pp \in \Pp(\R^n)$. Given any cube $Q\subset \R^n$, define $A_{\pp,Q}$ by 
\[A_{\pp,Q} (f) := \frac{ \| \ind_Q f\|_{\Lpp(\R^n)}}{\| \ind_Q \|_{\Lpp(\R^n)}},\]
for $f \in \Lpp_{\text{loc}}(\R^n)$. Define the maximal operator $M_\pp$ by 
\[M_\pp f(x) := \sup_{Q} A_{\pp,Q}(f)\, \ind_Q(x).\]
\end{definition}

We need the following results about $A_{\pp,Q}$ and $M_\pp$.
\begin{lemma}\cite[Inequality 
 2.5]{MR2837636}\label{MQDuality}
Let $\pp \in \Pp(\R^n)\cap LH(\R^n)$ with $p_+<\infty$. Then there exists a constant $C$ depending only on $[1]_{\calA_\pp}$ such that for any cube $Q\subset \R^n$, 
\[ \dashint_Q |f(x) g(x)|\dd x \leq C A_{\pp, Q} (f) A_{\cpp,Q}(g),\]
for all $f \in \Lpp_{\text{loc}}(\R^n)$ and $g \in \Lcpp_{\text{loc}}(\R^n)$.
\end{lemma}
\begin{theorem}\cite[Theorem 7.3.27]{diening_lebesgue_2011}\label{MqqBound}
Let $\pp, \qq, \rr \in \Pp(\R^n)\cap LH(\R^n)$ such that $\pp = \qq \rr$ and $r_- >1$. Then there exists a constant $C$ depending only on $r_-$ and the log-H\"{o}lder constants of $\pp$, $\qq$, $\rr$, such that
\[\| M_\qq f\|_{\Lpp(\R^n)} \leq C \| f\|_{\Lpp(\R^n)},\]
for all $f\in \Lpp(\R^n)$.
\end{theorem}

By tracking the constants in the proof of Theorem \ref{MqqBound}, we find that
\begin{align*}
    \|M_\qp\|_{\Lpp(\R^n)\to \Lpp(\R^n)} \leq C(n,C_\infty(1/\pp))\exp(2mC_\infty(1/t(\cdot))) \|M\|_{L^{r_-}(\R^n)\to L^{r_-}(\R^n)},
\end{align*}
where $M$ is the Hardy-Littlewood maximal operator, $t(\cdot)$ is given by
\[\frac{1}{t(x)}=\left| \frac{r_-}{r(x)} - \frac{r_-}{r_\infty}\right|,\]
and $m$ satisfies $\rho_{t(\cdot)}((e+|\cdot|)^{-m})<\infty$. As, 
\[\|M\|_{L^{r_-}(\R^n)\to L^{r_-}(\R^n)}\leq C(n) (r_-)',\]
we conclude that 
\begin{align*}
    \|M_\qp\|_{\Lpp(\R^n)\to \Lpp(\R^n)} \leq C(n,C_\infty(1/\pp))\exp(2mC_\infty(1/t(\cdot))) (r_-)'.
\end{align*}

Before moving on to prove Theorem \ref{thm:B}, we make a small digression to prove that if $\vf\in \Lpp(\R^n;\F^d)$, then $M'_{W,\pp}\vf(x)$ is finite for almost every $x\in \R^n$. The proof reduces down to the fact that $M_\pp f(x)$ is finite for almost every $x \in \R^n$. This fact seems to be known, but a proof appears to be missing from the literature. We provide it below.
\begin{prop}\label{MppFiniteAE}
    Given $\pp \in \Pp(\R^n)\cap LH(\R^n)$ with $p_+<\infty$, if $f\in \Lpp(\R^n)$, then $M_\pp f(x)<\infty$ for almost every $x\in \R^n$.
\end{prop}
\begin{proof}
Let $\pp \in \Pp(\R^n)\cap LH(\R^n)$ with $p_+<\infty$ and $f\in \Lpp(\R^n)$. We first show $\sup_{|Q|\geq 1} A_{\pp,Q}f(x)\leq \|f\|_{\Lpp(\R^n)}$ for all $x\in \R^n$. 

Fix $x\in \R^n$ and let $Q\subset \R^n$ be a cube containing $x$ such that $|Q|\geq 1$. Since $\pp \in LH(\R^n)$ by Lemma \ref{CharFunctionNormIneq},
\begin{align*}
    A_{\pp,Q} (f)\ind_Q(x) = \frac{\|f\ind_Q\|_{\Lpp(\R^n)}}{\|\ind_Q\|_{\Lpp(\R^n)}}
         \leq |Q|^{-\frac{1}{p_Q}} \|f\|_{\Lpp(\R^n)}
         \leq \|f\|_{\Lpp(\R^n)}.
\end{align*}

We next show for each $x\in \R^n$ with $f(x)<\infty$, there exists a constant $C$ such that
\[ \sup_{|Q|\leq 1} A_{\pp,Q}(f) \ind_Q(x)\leq |f(x)|+C.\]
Fix $x\in \R^n$ such that $f(x)<\infty$. Let $Q$ be a cube containing $x$ such that $|Q|\leq 1$. Observe that by the reverse triangle inequality,
\begin{align*}
    |A_{\pp,Q} (f)\ind_Q(x)- f(x)| & = \left| \frac{\|f\ind_Q\|_{\Lpp(\R^n)}}{\|\ind_Q\|_{\Lpp(\R^n)}} - \frac{\|f(x)\ind_Q(\cdot)\|_{\Lpp(\R^n)}}{\|\ind_Q\|_{\Lpp(\R^n)}}\right|\\
        & \leq \left\| \frac{f(\cdot) \ind_Q(\cdot)}{\|\ind_Q\|_{\Lpp(\R^n)}} - \frac{f(x)\ind_Q(\cdot)}{\|\ind_Q\|_{\Lpp(\R^n)}}\right\|_{\Lpp(\R^n)}.
\end{align*}

By Lemma \ref{CharFunctionNormIneq}, $\|\ind_Q\|_{\Lpp(\R^n)}\geq \frac{1}{6} |Q|^{\frac{1}{p_Q}}$. Since $|Q|\leq 1$ and $\pp \in LH_0(\R^n)$, by \cite[Lemma 2.8]{PenrodNormRH}, $|Q|^{-\frac{1}{p_Q}} \lesssim |Q|^{-\frac{1}{p(y)}}$ for all $y \in Q$. Thus, 
\begin{align*}
\left\| \frac{f(\cdot) \ind_Q(\cdot)}{\|\ind_Q\|_{\Lpp(\R^n)}} - \frac{f(x)\ind_Q(\cdot)}{\|\ind_Q\|_{\Lpp(\R^n)}}\right\|_{\Lpp(\R^n)} & \leq 6 \| |Q|^{-\frac{1}{p_Q}} | f(\cdot) - f(x)|\ind_Q(\cdot)\|_{\Lpp(\R^n)}\\
    & \lesssim \| |Q|^{-\frac{1}{\pp}}  |f(\cdot)-f(x)| \ind_Q(\cdot)\|_{\Lpp(\R^n)}.
\end{align*}

By \cite[Remark 2.83]{cruz-uribe_variable_2013}, 
\[ \lim_{\ell(Q)\to 0} \| |Q|^{-\frac{1}{\pp}}  |f(\cdot)-f(x)| \ind_Q(\cdot)\|_{\Lpp(\R^n)}=0.\]
Hence
\[\sup_{|Q|\leq 1} |A_{\pp,Q} (f) \ind_Q(x)-f(x)| = C<\infty.\]
Thus, $A_{\pp,Q}(f)\ind_Q(x)\leq |f(x)|+C$ for all cubes $Q$ with $|Q|\leq 1$. Since $f\in \Lpp(\R^n)$, by the definition of the $\Lpp(\R^n)$ norm, $f(x)$ must be finite for almost every $x\in \R^n$. Hence, $\sup_{|Q|\leq 1} A_{\pp,Q}(f)\ind_Q(x) <\infty$ for almost every $x\in \R^n$. Putting everything together, we have shown that
\[M_\pp f(x)=\sup_Q A_{\pp,Q}(f)\ind_Q(x)<\infty\]
for almost every $x\in \R^n$.
\end{proof}

As a corollary, we prove $M'_{W,\pp} \vf(x)$ is finite for almost every $x\in \R^n$.
\begin{corollary}\label{M'WppFiniteAE}
    Let $\pp \in \Pp(\R^n)\cap LH(\R^n)$ with $p_+<\infty$ and $W \in \calA_\pp$. If $\vf\in \Lpp(\R^n;\F^d)$, then $M'_{W,\pp} \vf(x)<\infty$ for almost every $x\in \R^n$.
\end{corollary}
\begin{proof}
Fix $x\in \R^n$ and let $Q\subset \R^n$ be a cube containing $x$. Let $\{\ve_i\}_{i=1}^d$ be the coordinate basis of $\F^d$. Then by Lemmas \ref{Holder}, \ref{SelfAdjointCommutes}, and \ref{opNorm:equiv}, the definition of the reducing operator $\overline{\calW}_Q^\cpp$, and Lemma \ref{CharFunctionNormIneq}, 
\begin{align*}
    \dashint_Q |\calW_Q^\pp W^{-1}(y)\vf(y)| \dd y & \leq \dashint_Q |\calW_Q^\pp W^{-1}(y) |_{\op} |\vf(y)|\dd y \\
        & \leq |Q|^{-1} \| |\calW_Q^\pp W^{-1}|_{\op}\ind_Q\|_{\Lcpp(\R^n)} \||\vf| \ind_Q\|_{\Lpp(\R^n)}\\
        & \leq |Q|^{-1} \||W^{-1}\calW_Q^\pp|_{\op} \ind_Q\|_{\Lcpp(\R^n)} \| |\vf|\ind_Q\|_{\Lpp(\R^n)}\\
        & \leq |Q|^{-1} \left\| \sum_{i=1}^d |W^{-1} \calW_Q^\pp \ve_i|\ind_Q(\cdot)\right\|_{\Lcpp(\R^n)} \| |\vf|\ind_Q\|_{\Lpp(\R^n)}\\
        & \leq |Q|^{-\frac{1}{p_Q}} \sum_{i=1}^d |Q|^{-\frac{1}{p'_Q}} \| |W^{-1}\calW_Q^\pp \ve_i|\ind_Q\|_{\Lcpp(\R^n)} \||\vf|\ind_Q\|_{\Lpp(\R^n)}\\
        & \leq |Q|^{-\frac{1}{p_Q}} \sum_{i=1}^d |\overline{\calW}_Q^\cpp \calW_Q^\pp\ve_i| \||\vf|\ind_Q\|_{\Lpp(\R^n)}\\
        & \leq d |Q|^{-\frac{1}{p_Q}} |\overline{\calW}_Q^\cpp \calW_Q^\pp|_{\op} \||\vf|\ind_Q\|_{\Lpp(\R^n)}\\
        & \leq 4 K_\pp [1]_{\calA_\pp} d [W]_{\calA_\pp}^R \frac{\||\vf|\ind_Q\|_{\Lpp(\R^n)}}{\|\ind_Q\|_{\Lpp(\R^n)}}.
\end{align*}
Thus, for all $x \in \R^n$, 
\[ M'_{W,\pp} \vf(x)\lesssim \sup_Q \frac{\||\vf|\ind_Q\|_{\Lpp(\R^n)}}{\|\ind_Q\|_{\Lpp(\R^n)}}\ind_Q(x) = \sup_Q A_{\pp,Q}(|\vf|)\ind_Q(x) = M_\pp |\vf|(x).  \]
Since $|\vf|\in \Lpp(\R^n)$, by Proposition \ref{MppFiniteAE}, $M_\pp |\vf|(x)$ is finite for almost every $x\in \R^n$. Hence, $M'_{W,\pp}\vf(x)<\infty$ for almost every $x\in \R^n$.
\end{proof}


\section{Proof of the main results}\label{sec:MainProof}
\subsection{Proof of Theorem~\ref{thm:B}}

We first need the following lemma.
\begin{lemma}\label{MaxOpWeightUnifBound}
Let $\pp \in \Pp(\R^n)\cap LH(\R^n)$ with $1 < p_- \leq p_+ <\infty$ and let $W:\R^n\to S_d$ be a matrix weight. If $W \in \calA_\pp$, then there exists $r>1$ such that if $\ucp = r \cpp$, then 
\[\sup_{Q\in \D} A_{\ucp,Q} (|W^{-1} \calW_Q^\pp|_{\op}) \leq C <\infty,\]
where the supremum is taken over all cubes in a given dyadic grid $\D$ in $\R^n$. In particular, the constant $C$ is an increasing function of $[W]_{\calA_\pp}$ that depends only on $n, d$, $p_-$, $p_+$, the $\log$-H\"older constants of $\pp$.
\end{lemma}
In the proof we will see that we can take 
\[
r= 1+\frac{1}{C_* (4[W^{-1}]_{\calA_\cpp})^{\left(1+2\frac{C_\infty(\cpp) (p')_+}{(p')_\infty (p')_-}\right)(p')_+}},
\]
where $C_*$ is as in Theorem~\ref{thm:NormRH}. By carefully tracking this constant in \cite{PenrodNormRH}, we find that it only depends on $n$, $p_+$, and the $\log$-H\"older constants of $\pp$.

\begin{proof}[Proof of Lemma~\ref{MaxOpWeightUnifBound}]
Let $\D$ be a dyadic grid of $\R^n$. For each cube $Q \in \D$ and $i=1,\ldots, d$, define the weight $w_{Q,i}$ by 
\[ w_{Q,i}(x) : =  |W^{-1}(x) \calW_Q^\pp \ve_i|.\]
By Lemma \ref{lem:Scalarization}, $w_{Q,i}$ is a scalar $\calA_\cpp$ weight for each $Q$ and $i$. Hence, by Theorem \ref{thm:NormRH}, for each $w_{Q,i}$, there exist constants $C_{Q,i}$ and $r_{Q,i}>1$ given by 
\[ C_{Q,i} = C^* [w_{Q,i}]_{\calA_\cpp}^{\frac{C_\infty(\cpp) (p')_+}{(p')_-^2}((p')_++1)} \]
and
\[ r_{Q,i} = 1+\frac{1}{C_* [w_{Q,i}]_{\calA_\cpp}^{\left(1+2\frac{C_\infty(\cpp) (p')_+}{(p')_\infty (p')_-}\right)(p')_+}}\]
such that for all cubes $P\subset \R^n$,
\[ |P|^{-\frac{1}{r_{Q,i} p'_P}} \| w_{Q,i} \ind_P\|_{L^{r_{Q,i}\cpp}(\R^n)} \leq C_{Q,i} |P|^{-\frac{1}{p'_P}} \| w_{Q,i}\ind_P\|_{\Lcpp(\R^n)}.\]

Define $r$ and $M$ by 
\[ r= 1+\frac{1}{C_* (4[W^{-1}]_{\calA_\cpp})^{\left(1+2\frac{C_\infty(\cpp) (p')_+}{(p')_\infty (p')_-}\right)(p')_+}}\]
and
\[M=C^* (4 [W^{-1}]_{\calA_\cpp})^{\frac{C_\infty(\cpp) (p')_+}{(p')_-^2}((p')_++1)}.\]
By Lemma \ref{lem:Scalarization}, $4 [W^{-1}]_{\calA_\cpp}\geq [w_{Q,i}]_{\calA_\cpp}$ for all $Q$ and $i$. Thus, $r\leq r_{Q,i}$ and $M \geq C_{Q,i}$ for all $Q$ and $i$. 
Define $\up$ by $\ucp = r \cpp$. Fix $Q \in \D$. By definition, we have $r <r_{Q,i}$, and so by Lemma \ref{CharFunctionNormIneq}, Lemma \ref{opNorm:equiv}, the triangle inequality, the norm reverse H\"{o}lder inequality (inequality \eqref{ineq:NormRH-s}), and Corollary \ref{cor:[1]_vpBound}, we have
\begin{align*}
\frac{\| |W^{-1}\calW_Q^\pp|_{\op}\ind_Q\|_{L^{\ucp}(\R^n)} }{\| \ind_Q\|_{L^{\ucp}(\R^n)}} & \leq 6|Q|^{-1/u'_Q} \| |W^{-1}\calW_Q^\pp|_{\op}\ind_Q\|_{L^{\ucp}(\R^n)}  \\
	&\leq 6|Q|^{-1/u'_Q} \left\| \sum_{i=1}^d |W^{-1}\calW_Q^\pp \ve_i|\ind_Q\right\|_{L^{\ucp}(\R^n)}\\
	& \leq  6 \sum_{i=1}^d|Q|^{-1/u'_Q} \| w_{Q,i} \ind_Q\|_{L^{\ucp}(\R^n)}\\
	& \lesssim 6\sum_{i=1}^d  C_{Q,i}|Q|^{-1/p'_Q} \left\| w_{Q,i}\ind_Q\right\|_{\Lcpp(\R^n)}\\
	& \leq 6M\sum_{i=1}^d |Q|^{-1/p'_Q} \| |W^{-1} \calW_Q^\pp \ve_i|\ind_Q\|_{\Lcpp(\R^n)}
\end{align*}
with the implicit constant depending only on $C_0(1/\cpp)$. Using the definition of the reducing operator $\overline{\calW}_Q^\cpp$, for each $i=1, \ldots, d$, we have
\[  |Q|^{-1/p'_Q}\| |W^{-1} \calW_Q^\pp \ve_i|\ind_Q\|_{\Lcpp(\R^n)}\leq |\overline{\calW}_Q^\cpp \calW_Q^\pp \ve_i|.\]

Combining this with the previous estimate and Lemma \ref{opNorm:equiv}, we get
\begin{align*}
\frac{ \| |W^{-1}\calW_Q^\pp|_{\op}\ind_Q\|_{L^{\ucp}(\R^n)}}{\|\ind_Q\|_{L^{\ucp}(\R^n)}}  
\leq C \sum_{i=1}^d |\overline{\calW}_Q^\cpp \calW_Q^\pp \ve_i| 
	 \leq C   |\overline{\calW}_Q^\cpp \calW_Q^\pp|_{\op} 
	 \leq C   [W]_{\calA_\pp}^R.
\end{align*}

Since $W \in \calA_\pp$, by Proposition \ref{prop:ApdotReducingOps}, $[W]_{\calA_\pp}^R <\infty$. Tracking the constants and applying Corollary \ref{cor:Symmetry}, we find that 
\begin{align*}
     & \sup_{Q\in \D} A_{\ucp,Q}(|W^{-1}\calW_Q^\pp|_{\op}) \\
     &\qquad\leq CC_D(1/\cpp) M [W]_{\calA_\pp}^R\\
     & \qquad= C C_D(1/\cpp)C^* [W^{-1}]_{\calA_\cpp}^{\frac{C_\infty(\cpp) (p')_+}{(p')_-^2}((p')_++1)}[W]_{\calA_\pp}^R\\
     & \qquad\leq C(n, d, (p')_-, (p')_+, C_D(1/\cpp), C_0(\cpp), C_\infty(\cpp) [W]_{\calA_\pp}^{1+{\frac{C_\infty(\cpp) (p')_+}{(p')_-^2}((p')_++1)}}.
\end{align*}
This completes the proof.
\end{proof}

\begin{proof}[Proof of Theorem~\ref{thm:B}]
Let $\pp \in \Pp(\R^n)\cap LH(\R^n)$ with $1 < p_- \leq p_+ < \infty$. Let $W \in \calA_\pp$. Fix $\vf \in \Lpp(\R^n;\F^d)$. We first make some reductions. By Lemma \ref{MaxOpFiniteSumBound}, we may reduce to a dyadic version of $M'_{W,\pp}$. To simplify notation, we will suppress any reference to the dyadic grid. 

By Lemma \ref{BddCompSuppReduction}, it suffices to assume $\vf \in L_c^\infty(\R^n;\F^d)$. To see this, let $\vf\in \Lpp(\R^n;\F^d)$ and define the sequence $\{\vf_k\}_{k=1}^\infty\subset L_c^\infty(\R^n;\F^d)$ by
\[
\vf_k(x):=\ind_{\{y\in B(0;k):|\vf(y)|\leq k\}}\vf(x).
\]
Then for a.e. $x\in\R^n$ we have 
\[
|\vf_k(x)\cdot\vu|=\ind_{\{y\in B(0;k):|\vf(y)|\leq k\}}(x)|\vf(x)\cdot\vu|\uparrow|\vf(x)\cdot\vu|
\]
as $k\to\infty$ for all $\vu\in\F^d$. Hence, by Lemma \ref{BddCompSuppReduction}, $M'_{W,\pp} \vf_k(x)$ increases to $M'_{W,\pp} \vf(x)$ for all $x\in \R^n$. Thus, if $M'_{W,\pp}$ is bounded from $L_c^\infty(\R^n;\F^d)$ to $\Lpp(\R^n;\F^d)$, then by Fatou's lemma for variable Lebesgue spaces (see \cite[Theorem 2.61]{cruz-uribe_variable_2013}), 
\begin{align*}
\| M'_{W,\pp}\vf\|_{\Lpp(\R^n)} & \leq \liminf_{k\to \infty} \| M'_{W,\pp} \vf_k\|_{\Lpp(\R^n)} \leq C \liminf_{k\to \infty} \| \vf_k\|_{\Lpp(\R^n;\F^d)}\\
&\leq C \| \vf\|_{\Lpp(\R^n;\F^d)}.
\end{align*}

By the homogeneity of the $\Lpp$ norm, we may assume $\|\vf\|_{\Lpp(\R^n;\F^d)} = 1$. Thus, we need to show $\|M'_{W,\pp}\vf\|_{\Lpp(\R^n)}\leq C<\infty$. Since $p_+<\infty$, by Lemma~\ref{ModNormEquiv}, it suffices to prove 
\[ \int_{\R^n} (M'_{W,\pp}\vf(x))^{p(x)}\dd x \leq C<\infty.\]

Choose $r$ from Lemma~\ref{MaxOpWeightUnifBound}. Define $\up$ by $\ucp = r \cpp$ so that
\[
\frac{p(\cdot)}{u(\cdot)}=\frac{1}{r'}p(\cdot)+\frac{1}{r}=:q(\cdot).
\]
Then we find that $q_-=\tfrac{1}{r'}p_-+\tfrac{1}{r}>\tfrac{1}{r'}+\tfrac{1}{r}=1$. As, $\qp \in LH(\R^n)$ by Lemma~\ref{lem:QuotientsLH} (with $\log$-H\"older constants depending only on those of $\pp$ and on $p_-$, $p_+$, and $r$), Theorem~\ref{MqqBound} implies that $M_\up$ is bounded on $\Lpp(\R^n)$.
For $k \in \Z$, define the sets $\Omega_k$ by
\[\Omega_k = \{x\in \R^n : M'_W \vf(x)>2^k\}.\]
By the Calderon-Zygmund decomposition from Lemma \ref{CZdecomp:AuxMaxOp}, for each $k$, there exists a collection $\{Q_j^k\}_j$ of dyadic cubes satisfying \eqref{CZdecomp:AuxMaxOp:decomp} and \eqref{CZdecomp:AuxMaxOp:ineq} with $\lambda =2^k$. Define the sets $E_j^k = Q_j^k \bs \Omega_{k+1}$. Since for each fixed $k$, the sets $Q_j^k$ are disjoint, we get that the sets $E_j^k$ are pairwise disjoint for all $k$ and $j$. Using inequality \eqref{CZdecomp:AuxMaxOp:ineq}, and Lemma \ref{MQDuality}, we get
\begin{align*}
\int_{\R^n} M'_{W,\pp} \vf(x)^{p(x)}  \dd x &  \leq \sum_{k\in \Z} \int_{\Omega_k \bs \Omega_{k+1}} 2^{(k+1)p(x)} \dd x\\ 
    & \leq 2^{p_+} \sum_{k,j} \int_{E_j^k} \left( \dashint_{Q_j^k} |\calW_{Q_j^k}^\pp W^{-1}(y) \vf(y)|\dd y \right)^{p(x)}\dd x\\
    & \leq 2^{p_+} \sum_{k,j} \int_{E_j^k} \left( \dashint_{Q_j^k} | \calW_{Q_j^k}^\pp W^{-1}(y)|_{\op} |\vf(y)|\dd y \right)^{p(x)}\dd x \\
    & \leq 2^{p_+} C^{p_+} \sum_{k,j} \int_{E_j^k} \left( A_{\ucp,Q_j^k} (|\calW_{Q_j^k}^\pp W^{-1}(y)|_{\op} ) A_{\up, Q_j^k} (|\vf|)\right)^{p(x)}\dd x.
\end{align*}
Note that the constant in the final estimate only depends on $[1]_{\calA_{u(\cdot)}}$, which, by Corollary~\ref{cor:[1]_vpBound} is bounded by a constant depending only on $n$, $p_-$, and the $\log$-H\"older constants of $\pp$. By Lemma \ref{SelfAdjointCommutes} and Lemma \ref{MaxOpWeightUnifBound}, we have
\[A_{\ucp,Q_j^k} (|\calW_{Q_j^k}^\pp W^{-1}|_{\op} ) \leq \sup_{Q\in\D} A_{\ucp,Q} (|W^{-1} \calW_Q^\pp|_{\op}) \leq C <\infty,\]
where $\D$ is the underlying dyadic grid. Combining this with the definition of $M_\up$, we get
\begin{align*}
& \sum_{k,j} \int_{E_j^k} \left( A_{\ucp,Q_j^k} (|\calW_{Q_j^k}^\pp W^{-1}(y)|_{\op} ) A_{\up, Q_j^k} (|\vf|)\right)^{p(x)}\dd x \\
    &\qquad \leq \sum_{k,j} \int_{E_j^k} C^{p(x)} M_{\up} (|\vf|)^{p(x)}\dd x\\
    & \qquad\leq C^{p_+} \int_{\R^n} M_\up (|\vf|)(x)^{p(x)}\dd x,
\end{align*}
with $C$ being the constant from Lemma \ref{MaxOpWeightUnifBound}. Since $M_\up$ is bounded on $\Lpp(\R^n)$ and $\|\vf\|_{\Lpp(\R^n;\F^d)}=1$, we have $\| M_{\up}(|\vf|)\|_{\Lpp(\R^n)}\leq C<\infty$. Applying the constant from Theorem \ref{MqqBound}, we get
\begin{align}\label{ineq:MupBound1}
\| M_{\up}\|_{\Lpp(\R^n)\to \Lpp(\R^n)} \leq (q_-)'C(n,C_\infty(1/\pp)e^{2mC_\infty(1/t(\cdot))}, 
\end{align}
where $t(x)$ is given by 
\[ \frac{1}{t(x)}=\left| \frac{q_-}{q(x)} - \frac{q_-}{q_\infty}\right|\]
and $m$ satisfies $\rho_{t(\cdot)} ((e+|\cdot|)^{-m})<\infty$. Thus, by Lemma \ref{ModNormEquiv}, $\int_{\R^n} M_\up(|\vf|)(x)^{p(x)} \dd x \leq C<\infty$. Thus, by our reductions
\[\| M'_{W,\pp}\vf\|_{\Lpp(\R^n)}\leq C\|\vf\|_{\Lpp(\R^n;\F^d)}.\]

To ensure that the constant $C$ above is given by $\phi([W]_{\calA_\pp})$, where $\phi$ is an increasing function, we need to estimate $(q_-)'$ and $e^{2mC_\infty(1/t(\cdot))}$. We have
\[
(q_-)'=\Big(1+\frac{1}{r'}(p_--1)\Big)'=1+\frac{r'}{p_--1}\eqsim (p_-)' r'.
\]
Our choice of $r$ from Lemma~\ref{MaxOpWeightUnifBound} yields
\[
r'\eqsim C_* [W^{-1}]_{\calA_\cpp}^{\left(1+2\frac{C_\infty(\cpp) (p')_+}{(p')_\infty (p')_-}\right)(p')_+}.
\]
By Corollary \ref{cor:Symmetry}, $[W^{-1}]_{\calA_\cpp}\approx_d [W]_{\calA_\pp}$. This proves the desired bound.

Next we estimate $C_\infty(1/t(\cdot))$. Note that by the definition of $t(\cdot)$, we must have $\tfrac{1}{t_\infty}=0$. Thus, for a.e. $x\in \R^n$, we have
\begin{align*}
\left|\frac{1}{t(x)}- \frac{1}{t_\infty}\right| & =  q_- \left|\frac{1}{q(x)}-\frac{1}{q_\infty}\right|
     = q_- \left| \frac{u(x)}{p(x)}-\frac{u_\infty}{p_\infty}\right|\\
    & \leq q_- \left( \left| \frac{u(x)}{p(x)} - \frac{u_\infty}{p(x)}\right| + \left| \frac{u_\infty}{p(x)}-\frac{u_\infty}{p_\infty}\right|\right)\\
    & \leq q_- \left( \frac{C_\infty(\up)}{p_-\log(e+|x|)} + \frac{u_\infty C_\infty(1/\pp)}{\log(e+|x|)}\right). 
\end{align*}
Note that also $q_-=\tfrac{1}{r'}p_-+\tfrac{1}{r}<(\tfrac{1}{r'}+\tfrac{1}{r})p_-=p_-$, and
\begin{align*}
    |u(x)-u_\infty| & = \left|\frac{u'(x)}{u'(x)-1}- \frac{u'_\infty}{u'_\infty-1}\right|
     = r \left| \frac{p'(x)}{rp'(x)-1}-\frac{p'_\infty}{rp'_\infty-1}\right|\\
    & = r \left| \frac{p'_\infty - p(x)}{(rp'(x)-1)(rp'_\infty-1)}\right|
     \leq \frac{r}{((p')_--1)^2} \frac{C_\infty(\cpp)}{\log(e+|x|)}.
\end{align*}
By the choice of $r$, $r<2$. Thus, 
\[ C_\infty(\up) \leq \frac{2C_\infty(\cpp)}{((p')_--1)^2}. \]
Lastly, since $r<2$,
\[
    u_\infty = \frac{u'_\infty}{u'_\infty-1}
     = \frac{rp'_\infty}{rp'_\infty-1}
     \leq \frac{rp'_\infty}{p'_\infty -1}
     \leq 2 p_\infty.
\]
Putting everything together, we get that the constant in \eqref{ineq:MupBound1} is bounded by an increasing function $\phi$ of $[W^{-1}]_{\calA_\cpp}$, which depends only on $n, d, p_-, p_+, p_\infty$, and the $\log$-H\"older constants of $\pp$.

\end{proof}

\subsection{Proof of Theorem~\ref{thm:A}}\label{sec:ProofOfA}
We first prove that the auxiliary operator $M'_{W,\pp}$ is equivalent to the modified auxiliary operator
\[
M''_{W,\pp}\vf(x)=\sup_Q\left(\dashint_Q\!|(\overline{\calW}^{p'(\cdot)}_Q)^{-1}W^{-1}(y)\vf(y)|\dd y\right)\ind_Q(x),
\]
which appears in \cite{KN24}.
\begin{proposition}\label{prop:auxiliarymaxopequivalence}
Let $\pp\in\Pp(\R^n)$ and $W\in\calA_\pp$. Then
\[
M''_{W,\pp}\vf(x)\lesssim_d M'_{W,\pp} \vf(x)\lesssim_d[W]_{\calA_\pp}M''_{W,\pp}\vf(x)
\]
for a.e. $x\in\R^n$ for all $\vf\in L^1_{\text{loc}}(\R^n;\F^d)$.
\end{proposition}
\begin{proof}
The second inequality follows from the fact that for all cubes $Q$ we have
\begin{align*}
|\calW_Q^\pp W^{-1}(x) \vf(x)|&\leq|\calW_Q^\pp\overline{\calW}^{p'(\cdot)}_Q|_{\text{op}}|(\overline{\calW}^{p'(\cdot)}_Q)^{-1}W^{-1}(x)\vf(x)|\\
&\leq[W]^R_{\calA_\pp}|(\overline{\calW}^{p'(\cdot)}_Q)^{-1}W^{-1}(x)\vf(x)|,
\end{align*}
combined with Lemma~\ref{prop:ApdotReducingOps}. For the first inequality, by an analogous argument, it suffices to bound $|(\overline{\calW}^{p'(\cdot)}_Q)^{-1}(\calW_Q^\pp)^{-1}|_{\text{op}}$ uniformly in $Q$. Indeed, for any $\vu\in\F^d$, by Lemmas \ref{SelfAdjointCommutes} and \ref{opNorm:equiv}, and the definition of the reducing operators $\overline{\calW}_Q^\cpp$ and $\calW_Q^\pp$, we have
\begin{align*}
|(\overline{\calW}^{p'(\cdot)}_Q)^{-1}&(\calW_Q^\pp)^{-1}\vu|
=\dashint_Q|(\overline{\calW}^{p'(\cdot)}_Q)^{-1}W^{-1}(y)W(y)(\calW_Q^\pp)^{-1}\vu|\dd y\\
& \leq \dashint_Q |(\overline{\calW}_Q^\cpp)^{-1} W^{-1}(y)|_{\op} |W(y) (\calW_Q^\pp)^{-1}\vu|\dd y\\
&\leq \sum_{i=1}^d\dashint_Q|W^{-1}(y)(\overline{\calW}^{p'(\cdot)}_Q)^{-1}\ve_i||W(y)(\calW_Q^\pp)^{-1}\vu|\dd y\\
&\lesssim \sum_{k=1}^d|Q|^{-\frac{1}{p_Q}}\||W^{-1}(\overline{\calW}^{p'(\cdot)}_Q)^{-1}\ve_i|\|_{L^{\cpp}(\R^n)}|Q|^{-\frac{1}{p_Q'}}\||W(\calW_Q^\pp)^{-1}\vu|\|_{L^\pp(\R^n)}\\
& \leq \sum_{i=1}^d |\overline{\calW}_Q^\cpp (\overline{\calW}_Q^\cpp)^{-1}\ve_i| |\calW_Q^\pp (\calW_Q^\pp)^{-1}\vu|\\
&=d|\vu|,
\end{align*}
proving the desired result.
\end{proof}

For the proof of the implications \ref{it:thmA3}$\Rightarrow$\ref{it:thmA1},\ref{it:thmA2} we require several tools from convex-set valued analysis. We refer the reader to the book \cite{AF09} for a general reference on this topic, or to \cite{bownik_extrapolation_2022} for a more specialized treatise. We let $\mc{K}$ denote the collection of closed non-empty subsets $K\subseteq\F^d$ satisfying:
\begin{itemize}
    \item \emph{Symmetry:} If $\vu\in K$, then also $\lambda \vu\in K$ for all $\lambda\in\F$ with $|\lambda|=1$;
    \item\emph{Convexity:} If $\vu,\vv\in K$, then $(1-t)\vu+t\vv\in K$ for all $0\leq t\leq 1$.
\end{itemize}
We say that a mapping $F:\R^n\to\mc{K}$ is measurable if for all open $E\subseteq\F^d$ the set
\[
F^{-1}(E):=\{x\in\R^n:F(x)\cap E\neq\emptyset\}
\]
is measurable. We denote the measurable mappings $F:\R^n\to\mc{K}$ by $L^0(\R^n;\mc{K})$. Moreover, we define the measurable selections of $F\in L^0(\R^n;\mc{K})$ by
\[
S^0(\R^n;F):=\{\vf\in L^0(\R^n;\F^d):\vf(x)\in F(x)\text{ a.e.}\}.
\]
For $\pp\in\Pp(\R^n)$ and a matrix weight $W:\R^n\to S_d$, we say that $F\in L^{p(\cdot)}_W(\R^n;\mc{K})$ if $S^0(\R^n;F)$ is a bounded set in $L_W^{p(\cdot)}(\R^n;\F^d)$, and set
\[
\|F\|_{L^{p(\cdot)}_W(\R^n;\mc{K})}:=\sup_{\vf\in S^0(\R^n;F)}\|\vf\|_{L_W^{p(\cdot)}(\R^n;\F^d)}.
\]
By \cite[Proposition~3.11]{Ni24b} we have $F\in L^{p(\cdot)}_W(\R^n;\mc{K})$ if and only if the function
\[
|W(x)F(x)|:=\sup_{\vu\in F(x)}|W(x)\vu|
\]
satisfies $|W(\cdot)F(\cdot)|\in L^{p(\cdot)}(\R^n)$. In this case, $\||W(\cdot)F(\cdot)|\|_{L^{p(\cdot)}(\R^n)}=\|F\|_{L^{p(\cdot)}_W(\R^n;\mc{K})}$.

For $F\in L^1(\R^n;\mc{K})$ and a measurable set $E\subseteq\R^n$ we define the \emph{Aumann integral}
\[
\int_E\!F\,\mathrm{d}x:=\Big\{\int_E\!\vf\,\mathrm{d}x:\vf\in S^0(\R^n;F)\Big\}.
\]
We say that $F\in L^1_{\text{loc}}(\R^n;\mc{K})$ if $\ind_QF\in L^1(\R^n;\mc{K})$ for all cubes $Q$ in $\R^n$, and we define
\[
\langle F\rangle_Q:=\frac{1}{|Q|}\int_Q\!F\,\mathrm{d}x.
\]
\begin{definition}
For $F\in L^1_{\loc}(\R^n;\mc{K})$ and a collection of cubes $\calQ$ in $\R^n$ we define $M_{\calQ}^{\mc{K}}F(x)$ as the smallest set in $\mc{K}$ containing
\[
\bigcup_{Q\in\calQ}\ind_Q(x)\langle F\rangle_Q.
\]
Given a matrix weight $W$ and $\vf\in L^0(\R^n;\F^d)$, we define
\[
M_{\calQ,W}f(x):=\sup_{Q\in\calQ}\langle |W(x)W^{-1}(\cdot)\vf|\rangle_Q\ind_Q(x).
\]
When $\calQ$ is the collection of all cubes in $\R^n$, we omit the subscript.
\end{definition}

The following result is an immediate corollary of \cite[Proposition~5.6]{Ni24b} applied to $X=L^\pp(\R^n)$.
\begin{prop}\label{prop:christgoldbergmax}
Let $\pp\in\Pp(\R^n)$ and let $W:\R^n\to S_d$ be a matrix weight. Then the following are equivalent:
\begin{enumerate}[(i)]
    \item\label{it:christgoldberg1} $M^{\mc{K}}:L^{p(\cdot)}_W(\R^n;\mc{K})\to L^{p(\cdot)}_W(\R^n;\mc{K})$;
    \item\label{it:christgoldberg2} $M_{W}:L^{p(\cdot)}(\R^n;\F^d)\to L^{p(\cdot)}(\R^n)$.
\end{enumerate}
Moreover, in this case we have
\[
\|M_{W}\|_{L^{p(\cdot)}(\R^n;\F^d)\to L^{p(\cdot)}(\R^n)}\eqsim_d\|M^{\mc{K}}\|_{L^{p(\cdot)}_W(\R^n;\mc{K})\to L^{p(\cdot)}_W(\R^n;\mc{K})}.
\]
\end{prop}

Finally, given a collection of cubes $\calQ$ in $\R^n$, we define the convex body operator 
\[
A^{\mc{K}}_{\calQ}F(x):=\sum_{Q\in\calQ}\ind_Q(x)\langle F\rangle_Q,
\]
where the sum is interpreted as a (possibly infinite) Minkowski sum. A collection of cubes $\mc{S}$ in $\R^n$ is called \emph{sparse} if there exists a pairwise disjoint collection of sets $(E_Q)_{Q\in\mc{S}}$ such that for each $Q\in\mc{S}$ we have $E_Q\subseteq Q$, $|E_Q|\geq\tfrac{1}{2}|Q|$. Then the following result is a variable exponent version of \cite[Theorem~6.10]{KN24} in the linear case.
\begin{theorem}\label{thm:sparsefromauxiliary}
Let $\pp\in\Pp(\R^n)$ and let $W\in\calA_\pp$. If
\[
M''_{W,\pp}:L^{p(\cdot)}(\R^n;\F^d)\to L^{p(\cdot)}(\R^n),\quad\text{and}\quad M''_{W^{-1},p'(\cdot)}:L^{p'(\cdot)}(\R^n;\F^d)\to L^{p'(\cdot)}(\R^n),
\]
then for all sparse collections $\mc{S}$ we have
\[
A^{\mc{K}}_{\mc{S}}:L^{p(\cdot)}_W(\R^n;\mc{K})\to L^{p(\cdot)}_W(\R^n;\mc{K})
\]
with
\begin{align*}
&\sup_{\mc{S}\text{ sparse}}\|A^{\mc{K}}_{\mc{S}}\|_{L^{p(\cdot)}_W(\R^n;\mc{K})\to L^{p(\cdot)}_W(\R^n;\mc{K})}\\
&\lesssim_d [W]_{\calA_\pp}\|M''_{W,p(\cdot)}\|_{L^{p(\cdot)}(\R^n;\F^d)\to L^{p(\cdot)}(\R^n)}\|M''_{W^{-1},p'(\cdot)}\|_{L^{p'(\cdot)}(\R^n;\F^d)\to L^{p'(\cdot)}(\R^n)}.
\end{align*}
\end{theorem}
For the proof, we require a lemma. For a matrix $A\in S_d$ and $F\in L^0(\R^n;\mc{K})$, we define $|AF(x)|:=\sup_{\vu\in F(x)}|A\vu|$.
\begin{lemma}\label{lem:sparsefromauxiliary}
Let $\pp\in\Pp(\R^n)$ and let $W:\R^n\to S_d$ be a matrix weight. For $F\in L^1_{\text{loc}}(\R^n;\mc{K})$, define
\[
M''_{W,\pp}F(x):=\sup_Q\left(\dashint_Q\!|(\overline{\calW}^{p'(\cdot)}_Q)^{-1}W^{-1}(y)F(y)|\dd y\right)\ind_Q(x).
\]
Then
\[
\|M''_{W,\pp}\|_{L^\pp(\R^n;\mc{K})\to L^\pp(\R^n)}\lesssim_d \|M''_{W,\pp}\|_{L^\pp(\R^n;\F^d)\to L^\pp(\R^n)}.
\]
\end{lemma}
\begin{proof}
As in the proof of \cite[Theorem~6.9]{bownik_extrapolation_2022}, for each $F\in L^{p(\cdot)}(\R^n;\mc{K})$ there is a constant $C(d)>0$ and mappings $\vf_1,\ldots,\vf_d\in S^0(\R^n;F)$ for which
\[
F(x)\subseteq C(d)\sum_{k=1}^d\mc{K}(\vf_k(x))
\]
for a.e. $x\in\R^n$, where $\mc{K}(\vu)$ denotes the smallest set in $\mc{K}$ containing $\vu\in\F^d$. Hence, for all cubes $Q$,
\[
\dashint_Q\!|(\overline{\calW}^{p'(\cdot)}_Q)^{-1}W^{-1}(y)F(y)|\dd y\leq C(d)\sum_{k=1}^d\left(\dashint_Q\!|(\overline{\calW}^{p'(\cdot)}_Q)^{-1}W^{-1}(y)\vf_k(y)|\dd y\right),
\]
proving that for a.e. $x\in\R^n$ we have
\[
M''_{W,\pp}F(x)\lesssim_d \sum_{k=1}^d M''_{W,\pp}\vf_k(x).
\]
This implies that
\begin{align*}
\|M''_{W,\pp}F\|_{L^\pp(\R^n)}&\lesssim_d\sum_{k=1}^d\|M''_{W,\pp}\vf_k\|_{L^\pp(\R^n)}\\
&\leq \|M''_{W,\pp}\|_{L^\pp(\R^n;\F^d)\to L^\pp(\R^n)}\sum_{k=1}^d\|\vf_k\|_{L^{p(\cdot)}(\R^n;\F^d)}\\
&\leq d\|M''_{W,\pp}\|_{L^\pp(\R^n;\F^d)\to L^\pp(\R^n)}\|F\|_{L^\pp(\R^n;\mc{K})},
\end{align*}
proving the desired assertion.
\end{proof}
\begin{proof}[Proof of Theorem~\ref{thm:sparsefromauxiliary}]
For a set $K\in\mc{K}$ and $\mathbf{u}\in\F^d$, we write $|K\cdot\mathbf{u}|:=\sup_{\mathbf{v}\in K}|\mathbf{v}\cdot\mathbf{u}|$. Then, for $F\in L^{p(\cdot)}_W(\R^n;\mc{K})$, $\vg\in S^0(\R^n;A_{\mc{S}}^{\mc{K}}F)$, and $\vh\in L^{p'(\cdot)}_{W^{-1}}(\R^n;\F^d)$, we have
\begin{align*}
\int_{\R^n}&\!|\vg(x)\cdot \vh(x)|\,\mathrm{d}x
\leq\sum_{Q\in\mc{S}}\int_Q\!|\langle F\rangle_Q\cdot \vh(x)|\,\mathrm{d}x\\
&\leq \sum_{Q\in\mc{S}}\int_Q\!|\mc{W}^{p(\cdot)}_Q \overline{\mc{W}}^{p'(\cdot)}_Q\langle (\overline{\mc{W}}^{p'(\cdot)}_Q)^{-1}F\rangle_Q\cdot (\mc{W}^{p(\cdot)}_Q)^{-1}\vh(x)|\,\mathrm{d}x\\
&\leq \sum_{Q\in\mc{S}}|\mc{W}^{p(\cdot)}_Q \overline{\mc{W}}^{p'(\cdot)}_Q|_{\text{op}}\langle |(\overline{\mc{W}}^{p'(\cdot)}_Q)^{-1}F|\rangle_Q \langle|(\mc{W}^{p(\cdot)}_Q)^{-1}\vh|\rangle_Q|Q|\\
&\leq[W]^R_{\calA_\pp}\sum_{Q\in\mc{S}}\langle |(\overline{\mc{W}}^{p'(\cdot)}_Q)^{-1}F|\rangle_Q \langle|(\mc{W}^{p(\cdot)}_Q)^{-1}\vh|\rangle_Q|Q|\\
&\lesssim [W]^R_{\calA_\pp}\sum_{Q\in\mc{S}}\int_{E_Q}\!M''_{W,p(\cdot)}(WF)M''_{W^{-1},p'(\cdot)}(W^{-1}\vh)\,\mathrm{d}x\\
&\leq [W]^R_{\calA_\pp}\int_{\R^n}\!M''_{W,p(\cdot)}(WF)M''_{W^{-1},p'(\cdot)}(W^{-1}\vh)\,\mathrm{d}x\\
&\lesssim [W]^R_{\calA_\pp}\|M''_{W,p(\cdot)}(WF)\|_{L^{p(\cdot)}(\R^n)}\|M''_{W^{-1},p'(\cdot)}(W^{-1}\vh)\|_{L^{p'(\cdot)}(\R^n)}.
\end{align*}
By Lemma~\ref{lem:sparsefromauxiliary} we have
\[
\|M''_{W,p(\cdot)}(WF)\|_{L^{p(\cdot)}(\R^n)}\lesssim_d\|M''_{W,p(\cdot)}\|_{L^{p(\cdot)}(\R^n)\to L^{p(\cdot)}(\R^n)}\|F\|_{L^{p(\cdot)}_W(\R^n;\mc{K})}.
\]
Thus, taking a supremum over all $h\in L^{p'(\cdot)}_{W^{-1}}(\R^n;\F^d)$ of norm $1$ and $g\in S^0(\R^n;A^{\mc{K}}_{\mc{S}}F)$, we conclude that
\begin{align*}
&\|A^{\mc{K}}_{\mc{S}}F\|_{L^{p(\cdot)}_W(\R^n;\mc{K})}\\
&\lesssim_d [W]^R_{\calA_\pp}\|M_{W,p(\cdot)}\|_{L^{p(\cdot)}(\R^n;\F^d)\to L^{p(\cdot)}(\R^n)}\|M_{W^{-1},p'(\cdot)}\|_{L^{p'(\cdot)}(\R^n;\F^d)\to L^{p'(\cdot)}(\R^n)}\|F\|_{L^{p(\cdot)}_W(\R^n;\mc{K})}.
\end{align*}
Since $[W]^R_{\calA_\pp}\eqsim_d[W]_{\calA_\pp}$ by Lemma~\ref{prop:ApdotReducingOps}, the result follows.
\end{proof}

Finally, we will need the following result, which is the $m=1$, $X=L^{p(\cdot)}(\R^n)$ case of \cite[Theorem~6.16]{KN24} combined with \cite[Theorem~4.1]{ConvOpsOnVLS}.
\begin{lemma}\label{lem:nondegeneracyvariablelebesgue}
Let $\pp\in\Pp(\R^n)$ and let $W:\R^n\to S_d$ be a matrix weight satisfying $|W^{-1}|_{\text{op}}\in L^{p'(\cdot)}_{\text{loc}}(\R^n)$. If $T$ is a directionally non-degenerate Calder\'on-Zygmund operator satisfying
\[
T:\Lpp_W(\R^n;\F^d)\to \Lpp_W(\R^n;\F^d),
\]
then $W\in\calA_\pp$ with
\[
[W]_{\calA_\pp}\lesssim_d\|T\|^2_{\Lpp_W(\R^n;\F^d)\to \Lpp_W(\R^n;\F^d)}.
\]
\end{lemma}

For the definition of directional non-degeneracy, we refer the reader to \cite{KN24}. The only result we will need is that the Riesz transform of the first coordinate
\[
R_1f(x):=\text{p.v.}\int_{\R^n}\!\frac{x_1-y_1}{|x-y|^{n+1}}f(y)\dd y
\]
is directionally non-degenerate, see \cite[Example~6.19]{KN24}. As a matter of fact, Goldberg's proof of \cite[Theorem~5.2]{goldberg_matrix_2003} shows that if a Calder\'on-Zygmund operator with convolution kernel $K$ is is non-degenerate in the sense of Stein \cite[p.210]{St93}, i.e., there is a unit vector $\vu\in\R^n$ for which for all $t\in\R$ we have
\[
K(t\vu)\gtrsim|t|^{-n},
\]
then it is also directionally non-degenerate.

\begin{proof}[Proof of Theorem~\ref{thm:A}]
The implication \ref{it:thmA1}$\Rightarrow$\ref{it:thmA3} follows from Lemma~\ref{lem:nondegeneracyvariablelebesgue} applied to $T=R_1$. Similarly, for \ref{it:thmA2}$\Rightarrow$\ref{it:thmA3}, note that the boundedness of the convex-set valued maximal operator $M^{\mc{K}}$ (which follows from Proposition~\ref{prop:christgoldbergmax}) implies the uniform boundedness of the family of averaging operators over the cubes in $\R^n$. Thus, the result follows from \cite[Theorem~4.1]{ConvOpsOnVLS}. It remains to prove \ref{it:thmA3}$\Rightarrow$\ref{it:thmA1},\ref{it:thmA2}.

To prove \ref{it:thmA3}$\Rightarrow$\ref{it:thmA2}, note that by Proposition~\ref{prop:christgoldbergmax} it suffices to bound $M^{\mc{K}}$. By the ``$1/3$" trick and monotone convergence (see \cite[Proposition~3.7]{Ni24b}), it suffices to bound $M^{\mc{K}}_{\mc{F}}$ for all finite collections $\mc{F}$ contained in a dyadic grid $\mc{D}$. Let $F\in L^\pp_W(\R^n;\mc{K})$. By \cite[Theorem~C]{Ni24b}, there is a sparse collection $\mc{S}\subseteq\mc{F}$ such that
\[
M^{\mc{K}}_{\mc{F}}F(x)\subseteq C(d) M_{\mc{S}}^{\mc{K}}F(x)\subseteq C(d) A^{\mc{K}}_{\mc{S}}F(x)
\]
for a.e. $x\in\R^n$. Since this implies that
\[
\|M^{\mc{K}}_{\mc{F}}F\|_{L^\pp_W(\R^n;\mc{K})}\lesssim_d\|A^{\mc{K}}_{\mc{S}}F\|_{L^\pp_W(\R^n;\mc{K})},
\]
the result follows from Theorem~\ref{thm:sparsefromauxiliary}.

Finally, for \ref{it:thmA3}$\Rightarrow$\ref{it:thmA1}, note that by \cite{NPTV17}, for each Calder\'on-Zygmund operator $T$ there is a constant $C_T>0$ such that for all $f\in L^\infty_c(\R^n;\F^d)$ there is a sparse collection $\mc{S}$ such that $Tf(x)\in C_TA_{\mc{S}}^\mc{K}(\mc{K}(f))(x)$ for a.e. $x\in\R^n$. Hence, by \cite[Proposition~3.6]{Ni24b}, we have
\[
\|Tf\|_{\Lpp_W(\R^n;\F^d)}\leq C_T\|A_{\mc{S}}^{\mc{K}}\|_{L^\pp_W(\R^n;\mc{K})\to L^\pp_W(\R^n;\mc{K})}\|f\|_{\Lpp_W(\R^n;\F^d)}.
\]
Thus, the result again follows from Theorem~\ref{thm:sparsefromauxiliary}.
\end{proof}

\subsection{Proof of Theorem~\ref{thm:C}} By \cite[Theorem~8.1]{Ni24b}, the assumptions of the theorem imply that if
\begin{equation}\label{eq:thmC1}
M_W:L^{p(\cdot)}(\R^n;\F^d)\to L^{p(\cdot)}(\R^n),\quad M_{W^{-1}}:L^{p'(\cdot)}(\R^n;\F^d)\to L^{p'(\cdot)}(\R^n),
\end{equation}
then $Tf$ is well-defined for all $\vf\in V$ with $S\vf\in \Lpp_W(\R^n;\F^d)$, and
\begin{equation}\label{eq:thmC2}
\begin{split}
&\|T\vf\|_{\Lpp_W(\R^n;\F^d)}\\
&\lesssim_d \phi(C(d)\|M_W\|_{L^{p(\cdot)}(\R^n;\F^d)\to L^{p(\cdot)}(\R^n)}^{\frac{1}{p_0'}}\|M_{W^{-1}}\|_{L^{p'(\cdot)}(\R^n;\F^d)\to L^{p'(\cdot)}(\R^n)}^{\frac{1}{p_0}})\|S\vf\|_{\Lpp_W(\R^n;\F^d)}.
\end{split}
\end{equation}
Since our assumptions on $\pp$ imply that we are in the setting of Theorem~\ref{thm:A}, we find that the first bound in \eqref{eq:thmC1} holds, and the associated operator norm is bounded by some increasing function of $[W]_{\calA_\pp}$, depending only on $n$, $d$, $p_-$, $p_+$, $p_\infty$, and the $LH(\R^n)$ constants of $\pp$. Since $[W]_{\calA_p}\eqsim_d[W^{-1}]_{\calA_{p'(\cdot)}}$ and $p'(\cdot)\in LH(\R^n)$ with the same constants as $\pp$, the same assertion is true for the second bound in \eqref{eq:thmC1}. Thus, the result follows from \eqref{eq:thmC2}.\hfill \qed

\section*{Acknowledgments}
The authors wish to thank the anonymous referee for their comments, corrections, and suggestions that have helped improve the overall quality and presentation of this work.


\bibliography{Bibliography.bib}{}
\bibliographystyle{alpha}

\end{document}